\documentclass{article}
\usepackage{amssymb}
\usepackage{amsmath}
\usepackage{amsfonts}

\newtheorem{theorem}{Theorem}

\newtheorem{corollary}[theorem]{Corollary}

\newtheorem{definition}[theorem]{Definition}

\newtheorem{lemma}[theorem]{Lemma}

\newtheorem{proposition}[theorem]{Proposition}
\newtheorem{remark}[theorem]{Remark}

\newenvironment{proof}[1][Proof]{\noindent\textbf{#1.} }{\ \rule{0.5em}{0.5em}}
\input{tcilatex}
\begin{document}

\title{High-jet Relations of the Heat Kernel\\
Embedding Map and Applications}
\author{Ke Zhu}
\maketitle

\begin{abstract}
For any compact Riemannian manifold $\left( M,g\right) $ and its heat kernel
embedding map $\psi _{t}:M\rightarrow l^{2}$ constructed in \cite{BBG}, we
study the higher derivatives of $\psi _{t}$ with respect to an orthonormal
basis at $x$ on $M$. \ As the heat flow time $t\rightarrow 0_{+}$, it turns
out the limiting angles between these derivative vectors are universal
constants independent on $g$, $x$ and the choice of orthonormal basis.
Geometric applications to the mean curvature and the Riemannian curvature
are given. Some algebraic structures of the $\infty $-jet space of $\psi
_{t} $ are explored.
\end{abstract}

\section{Introduction}

Let $\left( M,g\right) $ be a $n$-dimensional compact Riemannian manifold
with smooth metric $g$. In \cite{BBG}, B\'{e}rard, Besson and Gallot
considered the \emph{heat kernel embedding } 
\begin{equation}
\Phi _{t}:x\in M\rightarrow \left\{ e^{-\lambda _{j}t/2}\phi _{j}\left(
x\right) \right\} _{j\geq 1}\in l^{2}  \label{phi}
\end{equation}%
where $\lambda _{j}$ is the $j$-th eigenvalue of the Laplacian $\Delta _{g}$
of $\left( M,g\right) $, $\left\{ \phi _{j}\right\} _{j\geq 0}$ is the $%
L^{2} $ orthonormal eigenbasis of $\Delta _{g}$, and $l^{2}$ is the Hilbert
space of real-valued, square summable series. Let $\left\langle
,\right\rangle $ be the standard inner product of $l^{2}$, then $%
\left\langle \Phi _{t}\left( x\right) ,\Phi _{t}\left( y\right)
\right\rangle $ is the heat kernel 
\begin{equation*}
H\left( t,x,y\right) =\Sigma _{j=1}^{\infty }e^{-\lambda _{j}t}\phi
_{j}\left( x\right) \phi _{j}\left( y\right)
\end{equation*}%
on $M$. It was proved in \cite{BBG} that the \emph{normalized heat kernel
embedding\ }%
\begin{equation}
\psi _{t}:x\in M\rightarrow \sqrt{2}\left( 4\pi \right) ^{n/4}t^{\frac{n+2}{4%
}}\cdot \left\{ e^{-\lambda _{j}t/2}\phi _{j}\left( x\right) \right\}
_{j\geq 1}\in l^{2}  \label{psi}
\end{equation}%
is \emph{asymptotically isometric}, i.e. the induced metric of the embedded
image $\psi _{t}\left( M\right) $ in $l^{2}$ gets closer and closer to the
original metric $g$ as the heat flow time $t\rightarrow 0_{+}$. More
precisely, as $t\rightarrow 0_{+}$ 
\begin{equation}
\psi _{t}^{\ast }g_{0}=g+\frac{t}{3}\left( \frac{1}{2}S_{g}\cdot g-\text{Ric}%
_{g}\right) +O\left( t^{2}\right) ,  \label{asymp-isom-emb}
\end{equation}%
where $g_{0}$ is the standard metric in $l^{2}$, $S_{g}$ is the scalar
curvature, and Ric$_{g}$ is the Ricci curvature of $\left( M,g\right) $. The
embedding $\psi _{t}$ is \emph{canonical}, in the sense that it is
constructed by the eigenfunctions of the Laplacian of $\left( M,g\right) $.

Given any $x\in M$, we choose the normal coordinates $\left( x^{1},\cdots
x^{n}\right) $ near $x$ such that $\left\{ \frac{\partial }{\partial x^{j}}%
\right\} _{1\leq j\leq n}$ is \emph{orthonormal} at $x$. Let $%
\overrightarrow{\alpha }$ be the multi-index of the mixed derivative
operator $D^{\overrightarrow{\alpha }}$ on this chart. As $t\rightarrow
0_{+} $, it turns out these ($l^{2}$-valued) coefficients $\left\{ D^{%
\overrightarrow{\alpha }}\Phi _{t}\left( x\right) \right\} $ have a special
property: the angles between any two $D^{\overrightarrow{\alpha }}\Phi
_{t}\left( x\right) $ and $D^{\overrightarrow{\beta }}\Phi _{t}\left(
x\right) $ converge to universal constants \emph{independent} on the metric $%
g$, the point $x$ on $M$, and the orthonormal basis $\left\{ V_{i}\right\}
_{1\leq i\leq n}=\left\{ \frac{\partial }{\partial x^{j}}\right\} _{1\leq
j\leq n}$ at $x$, but only on $\overrightarrow{\alpha }$ and $%
\overrightarrow{\beta }$. The asymptote of the length of $D^{\overrightarrow{%
\alpha }}\Phi _{t}\left( x\right) $ only depends on $\overrightarrow{\alpha }
$ and $n$ too (Theorem \ref{high-jet-relation}). These \emph{asymptotic
relations }in the $k$-jet space of $\Phi _{t}$ (or $\psi _{t}$) have
interesting applications:

\begin{enumerate}
\item When $k=1$, \cite{BBG} used them to construct \emph{asymptotically
isometric embeddings} $\psi _{t}=\sqrt{2}\left( 4\pi \right) ^{n/4}t^{\frac{%
n+2}{4}}\Phi _{t}:M\rightarrow l^{2}$ as in $\left( \ref{asymp-isom-emb}%
\right) $. They also used the embedding $\psi _{t}$ to define the \emph{%
spectral distance} between different metrics on $M$. When the metrics have a
uniform upper bound of the diameter and a uniform lower bound of the Ricci
curvature, they established a \emph{precompactness} theorem on the space of
such metrics;

\item When $k=2$, by perturbing the almost isometric embedding $\psi _{t}\,$%
, \cite{WZ} constructed \emph{isometric embeddings} of compact Einstein
manifolds $M$ into $\mathbb{R}^{q\left( t\right) }$ with controlled second
fundamental form and $q\left( t\right) \sim t^{-n/2}$, where the $2$-jet
bundle $J^{2}\left( \psi _{t}\right) $ appeared in the linearized operator
of the isometric embedding problem. It was also showed that for any compact
Riemannian manifolds, the mean curvature vectors of the image $\frac{1}{%
\sqrt{t}}\psi _{t}\left( M\right) $ converge to constant length$\sqrt{\frac{%
n+2}{2n}}$ as $t\rightarrow 0_{+}$;

\item When $k=3$, we use them to show the embedded image $\frac{1}{\sqrt{t}}%
\psi _{t}\left( M\right) \subset l^{2}$ is \emph{asymptotically umbilical}
in the mean curvature vector direction as $t\rightarrow 0_{+}$ (Proposition %
\ref{umbilical}). We also propose to construct constant mean curvature
submanifolds in $\mathbb{R}^{q}$ by truncating and perturbing $\frac{1}{%
\sqrt{t}}\psi _{t}\left( M\right) \subset l^{2}$;

\item When $k=2$, but also considering the secondary leading terms in the
asymptote of the limiting angles, we prove the \textquotedblleft
asymptotic\textquotedblright\ Gauss formula\ (Proposition \ref{Asymp-Gauss})
to express the \emph{Riemannian curvature} tensor of $\left( M,g\right) $ as 
\begin{equation*}
R\left( X,Y,Z,W\right) =\lim_{t\rightarrow 0+}\left[ \left\langle \nabla
_{X}\nabla _{W}\psi _{t},\nabla _{Y}\nabla _{Z}\psi _{t}\right\rangle
-\left\langle \nabla _{X}\nabla _{Z}\psi _{t},\nabla _{Y}\nabla _{W}\psi
_{t}\right\rangle \right] .
\end{equation*}%
From this formula we can easily see the symmetry of the Riemannian curvature
tensor, including the second Bianchi identity (Lemma \ref{Bianchi-Identity}%
). We can also express the \emph{Levi-Civita connection} in terms of $\psi
_{t}$ (Lemma \ref{covariant-deri}).
\end{enumerate}

For example, the following $2$-jet relations were proved in \cite{WZ} and
played a crucial role to construct isometric embeddings of $M$ into $\mathbb{%
R}^{q}$. Let $\left\langle ,\right\rangle $ be the standard inner product in 
$l^{2}$ and $\left\vert \cdot \right\vert $ be the standard metric in $l^{2}$%
. We have

\begin{theorem}
\label{uni-lin-indep}(\cite{WZ} Theorem 2) For any $x\in M$, let $\left(
x^{1},\cdots x^{n}\right) $ be the normal coordinates near $x$ such that $%
\left\{ \frac{\partial }{\partial x^{j}}\right\} _{1\leq j\leq n}$ is \emph{%
orthonormal} at $x$. Then for the $n$ first derivatives vectors $\nabla
_{i}\Phi _{t}$ ($1\leq i\leq n$) and the $n\left( n+1\right) /2$ second
derivatives vectors $\nabla _{i}\nabla _{j}\Phi _{t}$ ($1\leq i\leq j\leq n$%
) of $\Phi _{t}$ with respect to $\left\{ \frac{\partial }{\partial x^{j}}%
\right\} _{1\leq j\leq n}$, as $t\rightarrow 0_{+}$ we have 
\begin{equation*}
\frac{\left\langle \nabla _{i}\Phi _{t}\left( x\right) ,\nabla _{j}\Phi
_{t}\left( x\right) \right\rangle }{\left\vert \nabla _{i}\Phi _{t}\left(
x\right) \right\vert \left\vert \nabla _{j}\Phi _{t}\left( x\right)
\right\vert }\rightarrow \delta _{ij},\ \text{and}\ \frac{\left\langle
\nabla _{i}\nabla _{j}\Phi _{t}\left( x\right) ,\nabla _{k}\Phi _{t}\left(
x\right) \right\rangle }{\left\vert \nabla _{i}\nabla _{j}\Phi _{t}\left(
x\right) \right\vert \left\vert \nabla _{k}\Phi _{t}\left( x\right)
\right\vert }\rightarrow 0,
\end{equation*}%
and for $i\neq j$ or $k\neq l$, 
\begin{equation*}
\frac{\left\langle \nabla _{i}\nabla _{j}\Phi _{t}\left( x\right) ,\nabla
_{k}\nabla _{l}\Phi _{t}\left( x\right) \right\rangle }{\left\vert \nabla
_{i}\nabla _{j}\Phi _{t}\left( x\right) \right\vert \left\vert \nabla
_{k}\nabla _{l}\Phi _{t}\left( x\right) \right\vert }\rightarrow 0,
\end{equation*}%
except when $\left\{ i,j\right\} =\left\{ k,l\right\} $ as sets.
Furthermore, for $i\neq j$,%
\begin{equation}
\frac{\left\langle \nabla _{i}\nabla _{i}\Phi _{t}\left( x\right) ,\nabla
_{j}\nabla _{j}\Phi _{t}\left( x\right) \right\rangle }{\left\vert \nabla
_{i}\nabla _{i}\Phi _{t}\left( x\right) \right\vert \left\vert \nabla
_{j}\nabla _{j}\Phi _{t}\left( x\right) \right\vert }\rightarrow \frac{1}{3}.
\label{uni-lin-independent}
\end{equation}%
The convergence is uniform for all $x$ on $M$.
\end{theorem}

\bigskip\ In this paper, we generalize Theorem \ref{uni-lin-indep} to obtain
the limiting angles between any two derivative vectors $D^{\overrightarrow{%
\alpha }}\Phi _{t}\left( x\right) $ and $D^{\overrightarrow{\beta }}\Phi
_{t}\left( x\right) $ as $t\rightarrow 0_{+}$. It is a refinement of
Proposition 15 in \cite{WZ}; The new observation is that the coefficients
there can be translated into certain enumeration problems of graphs and can
be solve explicitly (Proposition \ref{Graph-computation}).

In the following we denote the disjoint union of two sets by $\amalg $, and
let $\left\vert \overrightarrow{\alpha }\right\vert $ be the total degree of
the derivative operator $D^{\overrightarrow{\alpha }}$.

\begin{definition}
\label{constants} For any two derivative vectors $D^{\overrightarrow{\alpha }%
}\Phi _{t}\left( x\right) $ and $D^{\overrightarrow{\beta }}\Phi _{t}\left(
x\right) $, suppose there are $s$ distinct indices $\left\{
j_{1},j_{2},\cdots ,j_{s}\right\} $ in $\left\{ \overrightarrow{\alpha }%
\right\} \amalg \left\{ \overrightarrow{\beta }\right\} $, such that each $%
j_{r}$ has multiplicity $a_{r}$ in $\overrightarrow{\alpha }$, multiplicity $%
b_{r}$ in $\overrightarrow{\beta }$, and let the average multiplicity\emph{\ 
}$\sigma _{r}:=\frac{a_{r}+b_{r}}{2}$ for $r=1,2,\cdots ,s$. If all $%
a_{r}+b_{r}$ are even, we define the constants 
\begin{eqnarray}
A\left( \overrightarrow{\alpha },\overrightarrow{\beta }\right) &=&\left(
-1\right) ^{\frac{\left\vert \overrightarrow{\alpha }\right\vert -\left\vert 
\overrightarrow{\beta }\right\vert }{2}}{\LARGE \Pi }_{r=1}^{s}\left[ \frac{%
\left( 2\sigma _{r}\right) !}{2^{\sigma _{r}}\left( \sigma _{r}\right) !}%
\right] ,  \label{constant_A} \\
B\left( \overrightarrow{\alpha },\overrightarrow{\beta }\right) &=&\left.
A\left( \overrightarrow{\alpha },\overrightarrow{\beta }\right) \right/ %
\left[ A\left( \overrightarrow{\alpha },\overrightarrow{\alpha }\right)
A\left( \overrightarrow{\beta },\overrightarrow{\beta }\right) \right] ^{1/2}
\notag \\
&=&\left( -1\right) ^{\frac{\left\vert \overrightarrow{\alpha }\right\vert
-\left\vert \overrightarrow{\beta }\right\vert }{2}}{\LARGE \Pi }%
_{r=1}^{s}\left. \frac{\left( 2\sigma _{r}\right) !}{\sigma _{r}!}\right/ %
\left[ \frac{\left( 2a_{r}\right) !}{a_{r}!}\frac{\left( 2b_{r}\right) !}{%
b_{r}!}\right] ^{1/2}.  \label{constant_B}
\end{eqnarray}%
If some $a_{r}+b_{r}$ is odd, we simply let%
\begin{equation}
A\left( \overrightarrow{\alpha },\overrightarrow{\beta }\right) =0=B\left( 
\overrightarrow{\alpha },\overrightarrow{\beta }\right) .  \label{zero}
\end{equation}
\end{definition}

We have our main

\begin{theorem}
\label{high-jet-relation}(Asymptotic high-jet relations) For any two
derivative vectors $D^{\overrightarrow{\alpha }}\Phi _{t}\left( x\right) $
and $D^{\overrightarrow{\beta }}\Phi _{t}\left( x\right) $, as $t\rightarrow
0_{+}$,

\begin{enumerate}
\item If there is an index appeared with odd total multiplicity in $\left\{ 
\overrightarrow{\alpha }\right\} \amalg \left\{ \overrightarrow{\beta }%
\right\} $, then 
\begin{equation}
\frac{\left\langle D^{\overrightarrow{\alpha }}\Phi _{t}\left( x\right) ,D^{%
\overrightarrow{\beta }}\Phi _{t}\left( x\right) \right\rangle }{\left\vert
D^{\overrightarrow{\alpha }}\Phi _{t}\left( x\right) \right\vert \left\vert
D^{\overrightarrow{\beta }}\Phi _{t}\left( x\right) \right\vert }\rightarrow
0.  \label{ortho}
\end{equation}%
Especially this is the case when $\left\vert \overrightarrow{\alpha }%
\right\vert +\left\vert \overrightarrow{\beta }\right\vert $ is odd;

\item If each index appears with even total multiplicity in $\left\{ 
\overrightarrow{\alpha }\right\} \amalg \left\{ \overrightarrow{\beta }%
\right\} $, then 
\begin{equation}
\frac{\left\langle D^{\overrightarrow{\alpha }}\Phi _{t}\left( x\right) ,D^{%
\overrightarrow{\beta }}\Phi _{t}\left( x\right) \right\rangle }{\left\vert
D^{\overrightarrow{\alpha }}\Phi _{t}\left( x\right) \right\vert \left\vert
D^{\overrightarrow{\beta }}\Phi _{t}\left( x\right) \right\vert }\rightarrow
B\left( \overrightarrow{\alpha },\overrightarrow{\beta }\right) \neq 0;
\label{angle}
\end{equation}

\item We have 
\begin{equation}
\left\vert D^{\overrightarrow{a}}\Phi _{t}\left( x\right) \right\vert
^{2}\rightarrow \frac{1}{\left( 4\pi t\right) ^{n/2}}\left( \frac{1}{2t}%
\right) ^{\left\vert \overrightarrow{\alpha }\right\vert }A\left( 
\overrightarrow{\alpha },\overrightarrow{\alpha }\right) .  \label{length}
\end{equation}
\end{enumerate}

The convergence is uniform for all $x$ on $M$.
\end{theorem}

\bigskip For example, for $n\geq 2$, $i\neq j$, $\overrightarrow{\alpha }%
=\left( i,i\right) $ and $\overrightarrow{\beta }=\left( j,j\right) $, we
recover Theorem 2 in \cite{WZ}, because 
\begin{eqnarray*}
A\left( \overrightarrow{\alpha },\overrightarrow{\beta }\right) &=&{\LARGE %
\Pi }_{r=1}^{2}\left[ \left( -1\right) ^{1}\frac{2!}{2^{1}\cdot 1!}\right]
=1, \\
A\left( \overrightarrow{\alpha },\overrightarrow{\alpha }\right) &=&{\Large %
\Pi }_{r=1}^{1}\left[ \left( -1\right) ^{0}\frac{4!}{2^{2}\cdot 2!}\right]
=3=A\left( \overrightarrow{\beta },\overrightarrow{\beta }\right) , \\
B\left( \overrightarrow{\alpha },\overrightarrow{\beta }\right) &=&\left.
A\left( \overrightarrow{\alpha },\overrightarrow{\beta }\right) \right/ %
\left[ A\left( \overrightarrow{\alpha },\overrightarrow{\alpha }\right)
A\left( \overrightarrow{\beta },\overrightarrow{\beta }\right) \right]
^{1/2}=\frac{1}{3}.
\end{eqnarray*}

Theorem \ref{high-jet-relation} is a consequence of the more general

\begin{proposition}
\label{jet-relation}As $t\rightarrow 0_{+}$, we have%
\begin{equation}
\left( 4\pi t\right) ^{n/2}\left( 2t\right) ^{\left[ \frac{\left\vert 
\overrightarrow{\alpha }\right\vert +\left\vert \overrightarrow{\beta }%
\right\vert }{2}\right] }\cdot \left. D_{y}^{\overrightarrow{\beta }}D_{x}^{%
\overrightarrow{\alpha }}H\left( t,x,y\right) \right\vert _{x=y}=A\left( 
\overrightarrow{\alpha },\overrightarrow{\beta }\right) +O\left( t\right) ,
\label{leading-coeff}
\end{equation}%
where $\left[ c\right] $ is the largest integer less or equal to a given
real number $c$.
\end{proposition}

After the completion of this paper, the author received the paper \cite{Ni},
which is related to our results. The paper \cite{Ni}, motivated by
probabilistic questions, describes the small $\varepsilon $-asymptotics of
the jets along the diagonal of the integral kernel of the smoothing operator 
$w\left( \varepsilon \sqrt{\Delta _{g}}\right) $ for an arbitrary
nonnegative even Schwartz function $w$. For $w\left( x\right) =e^{-x^{2}}$,
these reduce to the short time asymptotics of the jets along the diagonal of
the heat kernel. 

The paper is organized as follows: In Section \ref{high-jet relations} we
first review the heat kernel and its off-diagonal expansion, then introduce
certain graphs to aid the computation of the leading terms in $\left( \ref%
{leading-coeff}\right) $, and then prove our main Theorem \ref%
{high-jet-relation}. In Section \ref{applications} we give applications of
the high-jet relations of $\psi _{t}$ to the mean curvature of $\psi
_{t}\left( M\right) \subset l^{2}$ and the Riemannian curvature of $M$. In
Section \ref{reformulations} we explore more algebraic structures of the
high-jet relations, and reformulate them by the lattice geometry on $\mathbb{%
Z}_{+}^{n}$, with the motivation for future applications.

\textbf{Acknowledgement.} The author thanks Xiaowei Wang for many
enlightenigng discussions during the collaboration of \cite{WZ}, where some
ideas of the current work emerged. The author thanks Liviu I. Nicolaescu for
pointing out the probabilistic point of view on our results. He also thanks
Clifford Taubes for his interest and support. The work of K. Zhu is
partially supported by the grant of Clifford Taubes from the National
Science Foundation.

\section{High-jet relations\label{high-jet relations}}

\subsection{Heat kernel and derivatives of the distance function}

Let $H\left( t,x,y\right) $ be the heat kernel of the Laplacian $\Delta _{g}$
on $\left( M,g\right) $,%
\begin{equation}
H\left( t,x,y\right) :=\Sigma _{j=1}^{\infty }e^{-\lambda _{j}t}\phi
_{j}\left( x\right) \phi _{j}\left( y\right) .  \label{heat-kernel}
\end{equation}%
It is well known that $H\left( t,x,y\right) $ has the \emph{%
Minakshisundaram-Pleijel expansion} 
\begin{equation*}
H\left( t,x,y\right) =\frac{1}{\left( 4\pi t\right) ^{n/2}}e^{-\frac{r^{2}}{%
4t}}U\left( t,x,y\right) 
\end{equation*}%
(e.g. \cite{BeGaM}, p.213, or \cite{Ch}, p.154), where $r=r\left( x,y\right) 
$ is the distance function for points $x$ and $y$ on $M$, 
\begin{equation}
U\left( t,x,y\right) =u_{0}\left( x,y\right) +tu_{1}\left( x,y\right)
+\cdots +t^{p}u_{p}\left( x,y\right) +O\left( t^{p+1}\right) 
\label{U-expansion}
\end{equation}%
as $t\rightarrow 0_{+}$, where the convergence is in $C^{l}$ for any $l\geq 0
$, and 
\begin{eqnarray}
u_{0}\left( x,y\right)  &=&\left[ \theta \left( x,y\right) \right] ^{-1/2}%
\text{ (for }x\text{ and }y\text{ close enough), }  \notag \\
u_{1}\left( x,x\right)  &=&\frac{S_{g}\left( x\right) }{6},  \label{u0u1}
\end{eqnarray}%
where 
\begin{equation}
\theta \left( x,y\right) =\frac{\text{volume density at }y\text{ read in the
normal coordinates around }x}{r^{n-1}}  \label{theta_xy}
\end{equation}%
with $r=r\left( x,y\right) $ (\cite{BeGaM}, p. 208). Given any $x$ on $M$,
let $\left( x^{1},\cdots ,x^{n}\right) $ be the normal coordinates around $x$
and $\left\{ V_{j}\right\} _{j=1}^{n}=\left\{ \frac{\partial }{\partial x^{j}%
}\right\} _{j=1}^{n}$ be an orthonormal basis at $x$. For any unit vector $%
V\in T_{x}M$, it was derived in \cite{BBG} that for $x_{s}=\exp _{x}\left(
sV\right) $ and $x_{\tau }=\exp _{x}\left( \tau V\right) $ with small $s$
and $\tau $, 
\begin{equation*}
\theta \left( x_{s},x_{\tau }\right) =1-\text{Ric}_{g}\left( \dot{x}\left(
s\right) ,\dot{x}\left( s\right) \right) \frac{\left( s-\tau \right) ^{2}}{3!%
}+O\left( \left\vert s-\tau \right\vert ^{3}\right) ,\text{ }
\end{equation*}%
so%
\begin{equation*}
u_{0}\left( x_{s},x_{\tau }\right) =1+\frac{1}{12}\text{Ric}_{g}\left( \dot{x%
}\left( s\right) ,\dot{x}\left( s\right) \right) \left( s-\tau \right)
^{2}+O\left( \left\vert s-\tau \right\vert ^{3}\right) .
\end{equation*}%
Hence 
\begin{eqnarray}
u_{0}\left( x,x\right)  &=&1,  \label{u0xx} \\
\nabla _{i}u_{0}\left( x,y\right) |_{x=y} &=&0,  \label{u0_deri} \\
\nabla _{i}U\left( t,x,y\right) |_{x=y} &=&O\left( t\right) ,  \label{U_deri}
\\
\nabla _{i}^{x}\nabla _{j}^{y}u_{0}\left( x,y\right) |_{x=y} &=&-\frac{1}{6}%
\text{Ric}_{g}\left( V_{i},V_{j}\right) ,  \label{u0_2nd_deri}
\end{eqnarray}

For later computations, we will need the following facts on the derivatives
of the squared distance function $r^{2}\left( x,y\right) $ on the diagonal $%
x=y$.

\begin{lemma}
\label{elimination-tools} For $r=r\left( x,y\right) $ on $M\times M$, let $%
\nabla _{i}=\nabla _{V_{j}}^{x}$ be the partial derivative with respect to $%
V_{j}$ in the $x$ variables, $\nabla _{\bar{j}}=\nabla _{V_{j}}^{y}$ be the
partial derivative with respect to $V_{j}$ in the $y$ variables, $\nabla
_{ij}=\nabla _{V_{i}}^{x}\nabla _{V_{j}}^{x}$, $\nabla _{i\overline{j}%
}=\nabla _{V_{i}}^{x}\nabla _{V_{j}}^{y}$, $\nabla _{ij\overline{k}}=$ $%
\nabla _{V_{i}}^{x}\nabla _{V_{j}}^{x}\nabla _{V_{k}}^{y}$ and so on. Then
we have the following identities 
\begin{eqnarray}
\nabla _{i}r^{2}\left( x,y\right) |_{x=y} &=&0=\nabla _{\bar{\imath}%
}r^{2}\left( x,y\right) |_{x=y},  \label{1st-r} \\
\nabla _{ij}r^{2}\left( x,y\right) |_{x=y} &=&2\delta _{ij}=-\nabla _{i%
\overline{j}}r^{2}\left( x,y\right) |_{x=y},  \label{2nd-r} \\
\nabla _{ijk}r^{2}\left( x,y\right) |_{x=y} &=&\nabla _{ij\overline{k}%
}r^{2}|_{x=y}=\nabla _{i\overline{j}\overline{k}}r^{2}|_{x=y}=\nabla _{%
\overline{i}\overline{j}\overline{k}}r^{2}|_{x=y}=0,  \label{3rd-r} \\
\nabla _{ijkl}r^{2}\left( x,y\right) |_{x=y} &=&-\nabla _{ij\overline{k}%
\overline{l}}r^{2}|_{x=y}\left( x,y\right) =-\frac{2}{3}\left(
R_{ikjl}\left( x\right) +R_{iljk}\left( x\right) \right) ,  \label{4th-r}
\end{eqnarray}%
where $R_{ijkl}\left( x\right) =R\left( V_{i},V_{j},V_{k},V_{l}\right)
\left( x\right) $ is the Riemannian curvature.
\end{lemma}

The identities $\left( \ref{1st-r}\right) $ and $\left( \ref{2nd-r}\right) $
are well-known (e.g. \cite{BBG}). The identities $\left( \ref{3rd-r}\right) $
and $\left( \ref{4th-r}\right) $ can be found in Chapter 16 (p.282) of \cite%
{De}, where $\frac{1}{2}r^{2}\left( x,y\right) $ is called the \emph{world
function}, and the limit at $x=y$ is called the coincidence limit.

\subsection{\protect\bigskip Admissible graphs}

To compute $\left\langle D^{\overrightarrow{\alpha }}\Phi _{t},D^{%
\overrightarrow{\beta }}\Phi _{t}\right\rangle $, it is useful to introduce
certain graphs.

\begin{definition}
(Vertex set) Let $\mathcal{S=}\left\{ \overrightarrow{\alpha }\right\}
\amalg \left\{ \overrightarrow{\beta }\right\} $ be the set with $\left\vert 
\overrightarrow{\alpha }\right\vert +\left\vert \overrightarrow{\beta }%
\right\vert $ vertices, where each vertex is an index $i\in \mathcal{S}$,
with the sign \textquotedblleft $+$\textquotedblright\ if it is in $\left\{ 
\overrightarrow{\alpha }\right\} $ and sign \textquotedblleft $-$%
\textquotedblright\ if it is in $\left\{ \overrightarrow{\beta }\right\} $.
We say the vertex has color $i$ if we do not distinguish its sign.
\end{definition}

\begin{definition}
(Admissible graph)\label{Graph-thry} A graph $G$ on the vertex set $\mathcal{%
S=}\left\{ \overrightarrow{\alpha }\right\} \amalg \left\{ \overrightarrow{%
\beta }\right\} $ is called \emph{admissible} if (i) Each vertex has valent
one; (ii) Each edge connects two vertices with the same color; (iii) Each
edge has a sign $\pm $ being the negative of the product of the signs of its
two end-points. Let Sign$\left( G\right) $ be the product of the signs of
all edges of $G$.
\end{definition}

\begin{remark}
The necessary and sufficient condition for $\mathcal{S=}\left\{ 
\overrightarrow{\alpha }\right\} \amalg \left\{ \overrightarrow{\beta }%
\right\} $ to have admissible graph(s) is that each index appears in $%
\mathcal{S}$ with even multiplicity. Especially $\left\vert \overrightarrow{%
\alpha }\right\vert +\left\vert \overrightarrow{\beta }\right\vert $ must be
even. If $G$ is not admissible, we simply let Sign$\left( G\right) =0$.
\end{remark}

\begin{proposition}
\label{Graph-computation}Suppose there are $s$ distinct indices $\left\{
j_{1},j_{2},\cdots ,j_{s}\right\} \subset \left\{ 1,2,\cdots ,n\right\} $ in 
$\left\{ \overrightarrow{\alpha }\right\} \amalg \left\{ \overrightarrow{%
\beta }\right\} $, such that each $j_{r}$ has multiplicity $a_{r}$ in $%
\overrightarrow{\alpha }$, and multiplicity $b_{r}$ in $\overrightarrow{%
\beta }$. Let $\sigma _{r}=\frac{a_{r}+b_{r}}{2}$. Then we have

\begin{enumerate}
\item For each admissible graph $G$ of $\mathcal{S=}\left\{ \overrightarrow{%
\alpha }\right\} \amalg \left\{ \overrightarrow{\beta }\right\} $, sign$%
\left( G\right) =\left( -1\right) ^{\frac{\left\vert \overrightarrow{\alpha }%
\right\vert -\left\vert \overrightarrow{\beta }\right\vert }{2}}$;

\item $\#\left\{ \text{admissible graphs of }\mathcal{S}\right\} ={\LARGE %
\Pi }_{r=1}^{s}\left[ \frac{\left( 2\sigma _{r}\right) !}{2^{\sigma
_{r}}\left( \sigma _{r}\right) !}\right] =\left\vert A\left( \overrightarrow{%
\alpha },\overrightarrow{\beta }\right) \right\vert $.
\end{enumerate}
\end{proposition}

\begin{proof}
By the definition of sign$\left( G\right) $, it is the product of the signs
of its edges. There are $\frac{\left\vert \overrightarrow{\alpha }%
\right\vert +\left\vert \overrightarrow{\beta }\right\vert }{2}$ edges, each
edge connecting vertices $i$ and $i^{\prime }$ has the sign $\left(
-1\right) \cdot $sign$\left( i\right) $sign$\left( i^{\prime }\right) $, and
the disjoint union of edges covers all vertices of $G$, so 
\begin{eqnarray*}
\text{sign}\left( G\right)  &=&\underset{\text{edge }e\text{ of }G}{\Pi }%
\text{sign}(e)=\left( -1\right) ^{\frac{\left\vert \overrightarrow{\alpha }%
\right\vert +\left\vert \overrightarrow{\beta }\right\vert }{2}}\cdot 
\underset{\text{vertex }i\text{ of }G}{\Pi }\text{sign}(i) \\
&=&\left( -1\right) ^{\frac{\left\vert \overrightarrow{\alpha }\right\vert
+\left\vert \overrightarrow{\beta }\right\vert }{2}}\cdot \left( -1\right)
^{\left\vert \overrightarrow{\beta }\right\vert }\text{ }=\left( -1\right) ^{%
\frac{\left\vert \overrightarrow{\alpha }\right\vert -\left\vert 
\overrightarrow{\beta }\right\vert }{2}}\text{,}
\end{eqnarray*}%
where the second line is because only vertices in $\overrightarrow{\beta }$
have \textquotedblleft $-$\textquotedblright\ sign. For each index $j_{r}$,
it has total multiplicity $2\sigma _{r}$ in $\mathcal{S=}\left\{ 
\overrightarrow{\alpha }\right\} \amalg \left\{ \overrightarrow{\beta }%
\right\} $, i.e. there are $2\sigma _{r}$ vertices with color $j_{r}$ in $%
\mathcal{S}\,$, so the ways to draw disjoint edges on this subset of
vertices in $\mathcal{S}$ is 
\begin{eqnarray*}
&&\left. \left( 
\begin{array}{c}
2\sigma _{r} \\ 
2\sigma _{r}-2%
\end{array}%
\right) \left( 
\begin{array}{c}
2\sigma _{r}-2 \\ 
2\sigma _{r}-4%
\end{array}%
\right) \cdots \left( 
\begin{array}{c}
2 \\ 
2%
\end{array}%
\right) \right/ \sigma _{r}! \\
&=&\left. \frac{\left( 2\sigma _{r}\right) !}{\left( 2\sigma _{r}-2\right)
!2!}\frac{\left( 2\sigma _{r}-2\right) !}{\left( 2\sigma _{r}-4\right) !2!}%
\cdots \frac{2!}{2!}\right/ \sigma _{r}! \\
&=&\frac{\left( 2\sigma _{r}\right) !}{2^{\sigma _{r}}\left( \sigma
_{r}\right) !}
\end{eqnarray*}%
Considering all distinct indices $j_{1},j_{2},\cdots $ and $j_{s}$ in $%
\mathcal{S}$, we see the number of admissible graphs on $\mathcal{S}$ is $%
{\LARGE \Pi }_{r=1}^{s}\left[ \frac{\left( 2\sigma _{r}\right) !}{2^{\sigma
_{r}}\left( \sigma _{r}\right) !}\right] $.
\end{proof}

\begin{remark}
The number $A\left( \overrightarrow{\alpha },\overrightarrow{\beta }\right) $
is related to the moments of the Gaussian $e^{-\left\vert x\right\vert ^{2}}$
by 
\begin{equation*}
A\left( \overrightarrow{\alpha },\overrightarrow{\beta }\right) =2^{\left[ 
\frac{\left\vert \overrightarrow{\alpha }\right\vert +\left\vert 
\overrightarrow{\beta }\right\vert }{2}\right] }\pi ^{-n/2}\int_{\mathbb{R}%
^{n}}e^{-\left\vert x\right\vert ^{2}}x^{\overrightarrow{\alpha }+%
\overrightarrow{\beta }}dx.
\end{equation*}%
Actually, the above combinatorics of the numbers $A\left( \overrightarrow{%
\alpha },\overrightarrow{\beta }\right) $ is a manifestation of a classical 
\emph{Wick formula} (or \emph{Isserlis' theorem}) of the Gauss integrals,
see e.g. \cite{AT}, Section 11.6. The formula is frequently used in the
Feynman integral. 
\end{remark}

\subsection{High-jet relations from the heat kernel expansion}

Given any $x$ on $M$, let $\left\{ V_{i}\right\} _{i=1}^{n}$ be an
orthonormal basis at $x$, $\nabla _{i}=\nabla _{V_{i}}^{x}$ be the partial
derivative with respect to $V_{i}$ in the $x$ variables, and $\nabla _{\bar{j%
}}=\nabla _{V_{j}}^{y}$ be the partial derivative with respect to $V_{j}$ in
the $y$ variables, where \textquotedblleft $-$\textquotedblright\ indicates
the derivative is with respect to the $y$-variables. We also denote $\nabla
_{ij}=\nabla _{V_{i}}^{x}\nabla _{V_{j}}^{x}$, $\nabla _{i\overline{j}%
}=\nabla _{V_{i}}^{x}\nabla _{V_{j}}^{y}$ and so on. It is easy to check
that for derivative operators $D^{\overrightarrow{\alpha }}$ and $D^{%
\overrightarrow{\beta }}$, 
\begin{equation}
\left\langle D^{\overrightarrow{\alpha }}\Phi _{t},D^{\overrightarrow{\beta }%
}\Phi _{t}\right\rangle \left( x\right) =D_{y}^{\overrightarrow{\beta }%
}D_{x}^{\overrightarrow{\alpha }}H\left( t,x,y\right) |_{x=y}.
\label{deri-angle-off-diagonal-formula}
\end{equation}

\begin{proposition}
\label{high-derivative} Let $\overrightarrow{\alpha }$ and $\overrightarrow{%
\beta }$ be two multi-indices in the set $\left\{ 1,2,\cdots n\right\} $.
Then as $t\rightarrow 0_{+}$, we have 
\begin{equation}
\left( 4\pi t\right) ^{n/2}\left( 2t\right) ^{\left[ \frac{\left\vert 
\overrightarrow{\alpha }\right\vert +\left\vert \overrightarrow{\beta }%
\right\vert }{2}\right] }\cdot \left. D_{y}^{\overrightarrow{\beta }}D_{x}^{%
\overrightarrow{\alpha }}H\left( t,x,y\right) \right\vert _{x=y}\rightarrow
A\left( \overrightarrow{\alpha },\overrightarrow{\beta }\right) ,
\label{DyDxH_asymp}
\end{equation}%
where $\left[ c\right] $ is the largest integer less or equal to a given
real number $c$. If $A\left( \overrightarrow{\alpha },\overrightarrow{\beta }%
\right) =0$, we further have 
\begin{equation}
\left( 4\pi t\right) ^{n/2}\left( 2t\right) ^{\left[ \frac{\left\vert 
\overrightarrow{\alpha }\right\vert +\left\vert \overrightarrow{\beta }%
\right\vert }{2}\right] }\cdot \left. D_{y}^{\overrightarrow{\beta }}D_{x}^{%
\overrightarrow{\alpha }}H\left( t,x,y\right) \right\vert _{x=y}=O\left(
t\right) .  \label{DyDxH_vanishing}
\end{equation}
\end{proposition}

\begin{proof}
For any multi-index $\vec{\gamma}$, applying the Leibniz rule to 
\begin{equation*}
H\left( t,x,y\right) =\frac{1}{\left( 4\pi t\right) ^{n/2}}e^{-\frac{r^{2}}{%
4t}}U\left( t,x,y\right)
\end{equation*}%
we can write 
\begin{equation*}
D^{\overrightarrow{\gamma }}H\left( t,x,y\right) =\frac{1}{\left( 4\pi
t\right) ^{n/2}}e^{-\frac{r^{2}}{4t}}P_{\overrightarrow{\gamma }}\left(
t,x,y\right) ,
\end{equation*}%
where $P_{\overrightarrow{\gamma }}\left( t,x,y\right) $ is a polynomial in $%
D^{\overrightarrow{\mu _{j}}}\left( -\frac{r^{2}\left( x,y\right) }{4t}%
\right) $ and $\left. D^{\overrightarrow{\eta }}U\left( t,x,y\right)
\right\vert _{x=y}$, i.e. each summand is of the form%
\begin{equation}
\left. \left( \Pi _{j=1}^{l}D^{\overrightarrow{\mu _{j}}}\left( -\frac{%
r^{2}\left( x,y\right) }{4t}\right) \right) \right\vert _{x=y}\cdot \left.
D^{\overrightarrow{\eta }}U\left( t,x,y\right) \right\vert _{x=y}
\label{summand}
\end{equation}%
for some multi-indices $\overrightarrow{\mu _{j}}$ and $\overrightarrow{\eta 
}$ with 
\begin{equation}
\Sigma _{j=1}^{l}\left\vert \overrightarrow{\mu _{j}}\right\vert +\left\vert 
\overrightarrow{\eta }\right\vert =\left\vert \overrightarrow{\gamma }%
\right\vert .  \label{total-degree}
\end{equation}%
For example, when $\overrightarrow{\gamma }=\partial _{x_{i}},$%
\begin{equation*}
P_{\overrightarrow{i}}\left( t,x,y\right) =\partial _{i}\left( -\frac{%
r^{2}\left( x,y\right) }{4t}\right) U\left( t,x,y\right) +\partial
_{i}U\left( t,x,y\right) .
\end{equation*}%
(more examples of $P_{\overrightarrow{\gamma }}\left( t,x,y\right) $ with $%
\left\vert \overrightarrow{\gamma }\right\vert \leq 4$ are in \cite{WZ}).
Now we take 
\begin{equation*}
\vec{\gamma}=\left( \overrightarrow{\alpha },\overrightarrow{\overline{\beta 
}}\right) ,\left\vert \overrightarrow{\gamma }\right\vert =\left\vert 
\overrightarrow{\alpha }\right\vert +\left\vert \overrightarrow{\beta }%
\right\vert ,D^{\vec{\gamma}}=D_{y}^{\overrightarrow{\beta }}D_{x}^{%
\overrightarrow{\alpha }},
\end{equation*}%
where $\overrightarrow{\overline{\beta }}$ indicates the derivative $D_{y}^{%
\overrightarrow{\beta }}$ is with respect to the $y$-variables. We have the
following claims:

\begin{enumerate}
\item As $t\rightarrow 0_{+}$, the nonzero summands of $P_{\vec{\gamma}%
}\left( t,x,y\right) $ involving the highest power of $\frac{1}{t}$ must
have $\overrightarrow{\mu _{j}}^{\prime }s$ with $\left\vert \overrightarrow{%
\mu _{j}}\right\vert =2$ as many as possible, such that each $%
\overrightarrow{\mu _{j}}$ is of the form $\left( i,i\right) $, $\left( i,%
\overline{i}\right) $ or $\left( \overline{i},\overline{i}\right) $ for some 
$i\in \left\{ 1,2,\cdots ,n\right\} $, and $\left\vert \overrightarrow{\eta }%
\right\vert =0$ or $1$. We also have 
\begin{equation*}
\left\vert D^{\overrightarrow{\gamma }}H\left( t,x,y\right)
|_{x=y}\right\vert \leq C\frac{1}{\left( 4\pi t\right) ^{n/2}}\cdot t^{-%
\left[ \frac{\left\vert \overrightarrow{\gamma }\right\vert }{2}\right] },
\end{equation*}%
where the constant $C$ depends on $n$ and $\left\vert \overrightarrow{\gamma 
}\right\vert $.

\item If the total multiplicity of each index in $\left\{ \overrightarrow{%
\alpha }\right\} \amalg \left\{ \overrightarrow{\beta }\right\} $ is even,
then we further have $l=\frac{\left\vert \overrightarrow{\gamma }\right\vert 
}{2}$, and $\left\vert \overrightarrow{\eta }\right\vert =0$ for the the
nonzero summands of $P_{\vec{\gamma}}\left( t,x,y\right) $ involving the
highest power of $\frac{1}{t}$. These terms are of the same sign $\left(
-1\right) ^{\frac{\left\vert \overrightarrow{\alpha }\right\vert -\left\vert 
\overrightarrow{\beta }\right\vert }{2}}$, and the total number of such
terms is $\left\vert A\left( \overrightarrow{\alpha },\overrightarrow{\beta }%
\right) \right\vert $.
\end{enumerate}

For Claim 1, from Lemma \ref{elimination-tools} we have 
\begin{eqnarray}
\partial _{i}\left( -\frac{r^{2}\left( x,y\right) }{4t}\right) |_{x=y} &=&0,
\notag \\
\text{ }\partial _{ij}\left( -\frac{r^{2}\left( x,y\right) }{4t}\right)
|_{x=y} &=&-\frac{\delta _{ij}}{2t}=-\partial _{i\overline{j}}\left( -\frac{%
r^{2}\left( x,y\right) }{4t}\right) |_{x=y},  \notag \\
\partial _{ijk}\left( -\frac{r^{2}\left( x,y\right) }{4t}\right) |_{x=y}
&=&0=\partial _{ij\overline{k}}\left( -\frac{r^{2}\left( x,y\right) }{4t}%
\right) |_{x=y},  \notag \\
\left\vert D^{\overrightarrow{\mu _{j}}}\left( -\frac{r^{2}\left( x,y\right) 
}{4t}\right) |_{x=y}\right\vert &\leq &C\cdot \frac{1}{t}\text{, for any }%
\vec{\mu}_{j}.  \label{t-power}
\end{eqnarray}%
So from the equations in $\left( \ref{t-power}\right) $ and the total degree
condition $\left( \ref{total-degree}\right) $, the summand 
\begin{equation*}
\left. \left( \Pi _{j=1}^{l}D^{\overrightarrow{\mu _{j}}}\left( -\frac{%
r^{2}\left( x,y\right) }{4t}\right) \right) \right\vert _{x=y}\cdot D^{%
\overrightarrow{\eta }}U
\end{equation*}%
can not exceed $\left[ \frac{\left\vert \overrightarrow{\gamma }\right\vert 
}{2}\right] $ copies of factor $\frac{1}{t}$. The $\left[ \frac{\left\vert 
\overrightarrow{\gamma }\right\vert }{2}\right] $ copies of factor $\frac{1}{%
t}$ can be achieved if $\left\vert \overrightarrow{\mu _{j}}\right\vert =2$
for all $j$, and $\left\vert \overrightarrow{\eta }\right\vert =0$ or $1$
depending on $\left\vert \overrightarrow{\gamma }\right\vert $ is even or
odd, so 
\begin{equation*}
\left\vert D^{\overrightarrow{\eta }}U\right\vert \leq \left\vert
u_{0}\left( x,x\right) \right\vert +\left\vert \nabla _{i}u_{0}\left(
x,y\right) |_{x=y}\right\vert \leq 2
\end{equation*}%
by $\left( \ref{u0xx}\right) $ and $\left( \ref{u0_deri}\right) $ as $%
t\rightarrow 0_{+}$. Such summand is nonzero if and only if each $%
\overrightarrow{\mu _{j}}$ is of the form $\left( i,i\right) $, $\left( i,%
\overline{i}\right) $ or $\left( \overline{i},\overline{i}\right) $ for some 
$i\in \left\{ 1,2,\cdots ,n\right\} $ by Lemma \ref{elimination-tools}, so
there are at most $3n$ choices of each $\overrightarrow{\mu _{j}}$ .
Therefore as $t\rightarrow 0_{+}$, 
\begin{eqnarray*}
\left\vert D^{\overrightarrow{\gamma }}H\left( t,x,y\right)
|_{x=y}\right\vert &=&\left\vert \left. \frac{1}{\left( 4\pi t\right) ^{n/2}}%
e^{-\frac{r^{2}}{4t}}P_{\overrightarrow{\gamma }}\left( t,x,y\right)
\right\vert _{x=y}\right\vert \\
&\leq &\frac{1}{\left( 4\pi t\right) ^{n/2}}\cdot \left( \frac{1}{2t}\right)
^{\left[ \frac{\left\vert \overrightarrow{\gamma }\right\vert }{2}\right]
}\cdot \left( 3n\right) ^{\left[ \frac{\left\vert \overrightarrow{\gamma }%
\right\vert }{2}\right] }\cdot 2.
\end{eqnarray*}

For Claim 2, since each index appears with even multiplicity in $\left\{ 
\overrightarrow{\alpha }\right\} \amalg \left\{ \overrightarrow{\beta }%
\right\} $, it is easy to see the summands $\left( \ref{summand}\right) $
with $l=\frac{\left\vert \overrightarrow{\gamma }\right\vert }{2}$ and $%
\left\vert \overrightarrow{\eta }\right\vert =0$ can have maximal number of $%
\overrightarrow{\mu _{j}}^{\prime }$s satisfying $\left\vert \overrightarrow{%
\mu _{j}}\right\vert =2$, with each $\overrightarrow{\mu _{j}}$ is of the
form $\left( i,i\right) $, $\left( i,\overline{i}\right) $ or $\left( 
\overline{i},\overline{i}\right) $ for some $i\in \left\{ 1,2,\cdots
,n\right\} $. It is of the form%
\begin{equation*}
\left. \left( \Pi _{j=1}^{l}D^{\overrightarrow{\mu _{j}}}\left( -\frac{%
r^{2}\left( x,y\right) }{4t}\right) \right) \right\vert _{x=y}\cdot U\left(
t,x,x\right) .
\end{equation*}%
Since 
\begin{equation}
\text{ }\partial _{ij}\left( -\frac{r^{2}\left( x,y\right) }{4t}\right)
|_{x=y}=-\frac{\delta _{ij}}{2t}=-\partial _{i\overline{j}}\left( -\frac{%
r^{2}\left( x,y\right) }{4t}\right) |_{x=y},  \label{basic-term}
\end{equation}%
the product 
\begin{equation*}
\left. \left( \Pi _{j=1}^{l}D^{\overrightarrow{\mu _{j}}}\left( -\frac{%
r^{2}\left( x,y\right) }{4t}\right) \right) \right\vert _{x=y}=\Pi _{j=1}^{l}%
\text{sign}\left( \vec{\mu}_{j}\right) \cdot \left( \frac{1}{2t}\right) ^{l}
\end{equation*}%
gives the highest power of $\frac{1}{t}$, where we let sign$\left( \vec{\mu}%
_{j}\right) =$ $1$ if $\overrightarrow{\mu _{j}}=\left( i,i\right) $ or $%
\left( \overline{i},\overline{i}\right) $, and sign$\left( \vec{\mu}%
_{j}\right) =$ $-1$ if $\overrightarrow{\mu _{j}}=\left( i,\overline{i}%
\right) $ . If these summands do not cancel each other, they give the
leading term in $D^{\overrightarrow{\gamma }}H\left( t,x,y\right) |_{x=y}$
as $t\rightarrow 0_{+}$, of the form 
\begin{equation*}
\frac{1}{\left( 4\pi t\right) ^{n/2}}\cdot \left( \frac{1}{2t}\right) ^{%
\frac{\left\vert \overrightarrow{\gamma }\right\vert }{2}}\cdot \left(
-1\right) ^{a}A,
\end{equation*}%
where $A$ is the number of such summands, and $\left( -1\right) ^{a}$ is
their common sign.

To determine the number $A$, we build a one-one correspondence between such
summands and admissible graphs in Definition \ref{Graph-thry}. For each $%
\overrightarrow{\mu _{j}}=\left( i,i\right) $, $\left( i,\overline{i}\right) 
$ or $\left( \overline{i},\overline{i}\right) $ of the summand, we draw an
edge between the corresponding vertices in $\mathcal{S=}\left\{ 
\overrightarrow{\alpha }\right\} \amalg \left\{ \overrightarrow{\beta }%
\right\} $. Drawing $l$ edges on $\mathcal{S}$ with no intersecting vertex,
we obtain an admissible graph $G$. The sign on each edge of $G$ reflects the
sign in $\left( \ref{basic-term}\right) $. Taking product of them we see sign%
$\left( G\right) $ is exactly the sign of the summand. Conversely, given any
admissible graph $G$, we can read $\overrightarrow{\mu _{j}}$ from its
edges, and then can write down the corresponding summand $\left. \left( \Pi
_{j=1}^{l}D^{\overrightarrow{\mu _{j}}}\left( -\frac{r^{2}\left( x,y\right) 
}{4t}\right) \right) \right\vert _{x=y}\cdot U\left( t,x,x\right) $ in $%
\left. D_{y}^{\overrightarrow{\beta }}D_{x}^{\overrightarrow{\alpha }%
}H\left( t,x,y\right) \right\vert _{x=y}$. From Proposition \ref%
{Graph-computation} we know sign$\left( G\right) =\left( -1\right) ^{\frac{%
\left\vert \overrightarrow{\alpha }\right\vert -\left\vert \overrightarrow{%
\beta }\right\vert }{2}}$, so the summands are of the same sign $\left(
-1\right) ^{\frac{\left\vert \overrightarrow{\alpha }\right\vert -\left\vert 
\overrightarrow{\beta }\right\vert }{2}}$. Proposition \ref%
{Graph-computation} also gives the signed count of admissible graphs, which
is $A\left( \overrightarrow{\alpha },\overrightarrow{\beta }\right) $. Claim
2 is proved.

If each index appears with even multiplicity in $\left\{ \overrightarrow{%
\alpha }\right\} \amalg \left\{ \overrightarrow{\beta }\right\} $, then $%
\left\vert \overrightarrow{\gamma }\right\vert $ is even, $\left[ \frac{%
\left\vert \overrightarrow{\gamma }\right\vert }{2}\right] =\frac{\left\vert 
\overrightarrow{\gamma }\right\vert }{2}$. From the above argument we have
as $t\rightarrow 0_{+}$,%
\begin{eqnarray*}
&&\left( 4\pi t\right) ^{n/2}\left( 2t\right) ^{\left[ \frac{\left\vert 
\overrightarrow{\alpha }\right\vert +\left\vert \overrightarrow{\beta }%
\right\vert }{2}\right] }\cdot \left. D_{y}^{\overrightarrow{\beta }}D_{x}^{%
\overrightarrow{\alpha }}H\left( t,x,y\right) \right\vert _{x=y} \\
&\rightarrow &\left( 4\pi t\right) ^{n/2}\left( 2t\right) ^{\frac{\left\vert 
\overrightarrow{\gamma }\right\vert }{2}}\cdot \frac{1}{\left( 4\pi t\right)
^{n/2}}\cdot \left( \frac{1}{2t}\right) ^{\frac{\left\vert \overrightarrow{%
\gamma }\right\vert }{2}}A\left( \overrightarrow{\alpha },\overrightarrow{%
\beta }\right) =A\left( \overrightarrow{\alpha },\overrightarrow{\beta }%
\right) .
\end{eqnarray*}

If some index $i$ appears with odd multiplicity in $\left\{ \overrightarrow{%
\alpha }\right\} \amalg \left\{ \overrightarrow{\beta }\right\} $, then by
our definition $A\left( \overrightarrow{\alpha },\overrightarrow{\beta }%
\right) =0$. For a nonzero summand of the form 
\begin{equation*}
\left. \left( \Pi _{j=1}^{l}D^{\overrightarrow{\mu _{j}}}\left( -\frac{%
r^{2}\left( x,y\right) }{4t}\right) \right) \right\vert _{x=y}\cdot D^{%
\overrightarrow{\eta }}U,
\end{equation*}
by the equations in $\left( \ref{t-power}\right) $ and the total degree
condition $\left( \ref{total-degree}\right) $, it can have the $\frac{1}{t}$
power at most $\left[ \frac{\left\vert \overrightarrow{\gamma }\right\vert -1%
}{2}\right] $. If the power $\left[ \frac{\left\vert \overrightarrow{\gamma }%
\right\vert -1}{2}\right] $ $\ $is achieved, then $\left\vert \vec{\eta}%
\right\vert \leq 2$ by the total degree condition $\left( \ref{total-degree}%
\right) $, and $\left\vert \vec{\eta}\right\vert =2$ is only possible when $%
\left\vert \overrightarrow{\gamma }\right\vert $ is even. If $\left\vert 
\vec{\eta}\right\vert =0$, there will be a term $D^{\overrightarrow{\sigma }%
}\left( r^{2}\right) |_{x=y}$ in the summand, such that $\left\vert \vec{%
\sigma}\right\vert \leq 3$ and $i$ appears with odd multiplicity in $\vec{%
\sigma}$, making the summand zero by Lemma \ref{elimination-tools}. So we
must have $\left\vert \vec{\eta}\right\vert =1$ when $\left\vert 
\overrightarrow{\gamma }\right\vert $ is odd.

Similar to the argument in Claim 2, we have%
\begin{equation}
\left\vert \left. D^{\overrightarrow{\gamma }}H\left( t,x,y\right)
\right\vert _{x=y}\right\vert \leq C\frac{1}{\left( 4\pi t\right) ^{n/2}}%
\cdot \left( \frac{1}{t}\right) ^{\left[ \frac{\left\vert \overrightarrow{%
\gamma }\right\vert -1}{2}\right] }\left\vert D^{\overrightarrow{\eta }%
}U\right\vert .  \label{odd-decrease-power}
\end{equation}%
When $\left\vert \vec{\eta}\right\vert =1$, by $\left( \ref{U_deri}\right) $
we have $\left\vert D^{\overrightarrow{\eta }}U\right\vert \leq Ct$.
Therefore 
\begin{eqnarray*}
&&\left( 4\pi t\right) ^{n/2}\left( 2t\right) ^{\left[ \frac{\left\vert 
\overrightarrow{\alpha }\right\vert +\left\vert \overrightarrow{\beta }%
\right\vert }{2}\right] }\cdot \left. D_{y}^{\overrightarrow{\beta }}D_{x}^{%
\overrightarrow{\alpha }}H\left( t,x,y\right) \right\vert _{x=y} \\
&\rightarrow &\left\{ 
\begin{array}{c}
t^{\left[ \frac{\left\vert \overrightarrow{\gamma }\right\vert }{2}\right]
}\cdot O\left( t^{-\left[ \frac{\left\vert \overrightarrow{\gamma }%
\right\vert -1}{2}\right] }\right) \cdot O\left( t\right) =O\left( t\right)
\rightarrow 0\text{ (if }\left\vert \overrightarrow{\gamma }\right\vert 
\text{ is odd),} \\ 
t^{\left[ \frac{\left\vert \overrightarrow{\gamma }\right\vert }{2}\right]
}\cdot O\left( t^{-\left[ \frac{\left\vert \overrightarrow{\gamma }%
\right\vert -1}{2}\right] }\right) \cdot O\left( 1\right) =O\left( t\right)
\rightarrow 0\text{ (if }\left\vert \overrightarrow{\gamma }\right\vert 
\text{ is even).}%
\end{array}%
\right\} =A\left( \overrightarrow{\alpha },\overrightarrow{\beta }\right) 
\text{.}
\end{eqnarray*}%
Proposition \ref{high-derivative} is proved.
\end{proof}

Now we are ready to give the proof of Theorem \ref{high-jet-relation}.

\begin{proof}
Letting $\overrightarrow{\alpha }=\overrightarrow{\beta }$ in Proposition %
\ref{high-derivative}, and noticing 
\begin{equation*}
\left\langle D^{\overrightarrow{\alpha }}\Phi _{t},D^{\overrightarrow{\beta }%
}\Phi _{t}\right\rangle \left( x\right) =D_{y}^{\overrightarrow{\beta }%
}D_{x}^{\overrightarrow{\alpha }}H\left( t,x,y\right) |_{x=y},
\end{equation*}
we have%
\begin{eqnarray*}
&&\left( 4\pi t\right) ^{n/2}\left( 2t\right) ^{\left\vert \overrightarrow{%
\alpha }\right\vert }\cdot \left\vert D^{\overrightarrow{\alpha }}\Phi
_{t}\left( x\right) \right\vert ^{2} \\
&=&\left( 4\pi t\right) ^{n/2}\left( 2t\right) ^{\left[ \frac{\left\vert 
\overrightarrow{\alpha }\right\vert +\left\vert \overrightarrow{\alpha }%
\right\vert }{2}\right] }\cdot \left. D_{y}^{\overrightarrow{\alpha }}D_{x}^{%
\overrightarrow{\alpha }}H\left( t,x,y\right) \right\vert _{x=y}\rightarrow
A\left( \overrightarrow{\alpha },\overrightarrow{\alpha }\right) ,
\end{eqnarray*}%
as $t\rightarrow 0_{+}$, i.e. 
\begin{equation}
\left\vert D^{\overrightarrow{\alpha }}\Phi _{t}\left( x\right) \right\vert
^{2}\rightarrow \frac{1}{\left( 4\pi t\right) ^{n/2}}\left( \frac{1}{2t}%
\right) ^{\left\vert \overrightarrow{\alpha }\right\vert }A\left( 
\overrightarrow{\alpha },\overrightarrow{\alpha }\right) .
\label{norm-limit}
\end{equation}

If $\left\vert \overrightarrow{\alpha }\right\vert +\left\vert 
\overrightarrow{\beta }\right\vert $ is even, by Proposition \ref%
{high-derivative} we have 
\begin{eqnarray*}
&&\lim_{t\rightarrow 0_{+}}\frac{\left\langle D^{\overrightarrow{\alpha }%
}\Phi _{t}\left( x\right) ,D^{\overrightarrow{\beta }}\Phi _{t}\left(
x\right) \right\rangle }{\left\vert D^{\overrightarrow{\alpha }}\Phi
_{t}\left( x\right) \right\vert \left\vert D^{\overrightarrow{\beta }}\Phi
_{t}\left( x\right) \right\vert } \\
&=&\lim_{t\rightarrow 0_{+}}\frac{\left( 4\pi t\right) ^{n/2}\left(
2t\right) ^{\frac{\left\vert \overrightarrow{\alpha }\right\vert +\left\vert 
\overrightarrow{\beta }\right\vert }{2}}\cdot \left\langle D^{%
\overrightarrow{\alpha }}\Phi _{t}\left( x\right) ,D^{\overrightarrow{\beta }%
}\Phi _{t}\left( x\right) \right\rangle }{\left( \left( 4\pi t\right)
^{n/4}\left( 2t\right) ^{\frac{\left\vert \overrightarrow{\alpha }%
\right\vert }{2}}\left\vert D^{\overrightarrow{\alpha }}\Phi _{t}\left(
x\right) \right\vert \right) \cdot \left( \left( 4\pi t\right) ^{n/4}\left(
2t\right) ^{\frac{\left\vert \overrightarrow{\beta }\right\vert }{2}%
}\left\vert D^{\overrightarrow{\beta }}\Phi _{t}\left( x\right) \right\vert
\right) } \\
&=&\left. \left[ A\left( \overrightarrow{\alpha },\overrightarrow{\beta }%
\right) \right] \right/ \left[ A\left( \overrightarrow{\alpha },%
\overrightarrow{\alpha }\right) A\left( \overrightarrow{\beta },%
\overrightarrow{\beta }\right) \right] ^{1/2}=B\left( \overrightarrow{\alpha 
},\overrightarrow{\beta }\right) .
\end{eqnarray*}

If $\left\vert \overrightarrow{\alpha }\right\vert +\left\vert 
\overrightarrow{\beta }\right\vert $ is odd, then some index must appear
with odd multiplicity in $\left\{ \overrightarrow{\alpha }\right\} \amalg
\left\{ \overrightarrow{\beta }\right\} $. By $\left( \ref%
{odd-decrease-power}\right) $ and $\left( \ref{norm-limit}\right) $,%
\begin{eqnarray*}
&&\lim_{t\rightarrow 0_{+}}\frac{\left\vert \left\langle D^{\overrightarrow{%
\alpha }}\Phi _{t}\left( x\right) ,D^{\overrightarrow{\beta }}\Phi
_{t}\left( x\right) \right\rangle \right\vert }{\left\vert D^{%
\overrightarrow{\alpha }}\Phi _{t}\left( x\right) \right\vert \left\vert D^{%
\overrightarrow{\beta }}\Phi _{t}\left( x\right) \right\vert } \\
&\leq &\lim_{t\rightarrow 0_{+}}\left. C\frac{1}{\left( 4\pi t\right) ^{n/2}}%
\left( \frac{1}{t}\right) ^{\left[ \frac{\left\vert \overrightarrow{\alpha }%
\right\vert +\left\vert \overrightarrow{\beta }\right\vert -1}{2}\right]
}\right/ \left( \frac{1}{\left( 4\pi t\right) ^{n/4}}\left( \frac{1}{2t}%
\right) ^{\frac{\left\vert \overrightarrow{\alpha }\right\vert }{2}}A\left( 
\overrightarrow{\alpha },\overrightarrow{\alpha }\right) \cdot \frac{1}{%
\left( 4\pi t\right) ^{n/4}}\left( \frac{1}{2t}\right) ^{\frac{\left\vert 
\overrightarrow{\beta }\right\vert }{2}}A\left( \overrightarrow{\beta },%
\overrightarrow{\beta }\right) \right) \\
&=&O\left( t\right) \rightarrow 0=B\left( \overrightarrow{\alpha },%
\overrightarrow{\beta }\right) \text{.}
\end{eqnarray*}%
Theorem \ref{high-jet-relation} is proved.
\end{proof}

\begin{remark}
\label{switch-ordering}If we are only concerned with the leading order terms
in $\left\langle D^{\overrightarrow{\alpha }}\Phi _{t},D^{\overrightarrow{%
\beta }}\Phi _{t}\right\rangle \left( x\right) $ or $D_{y}^{\overrightarrow{%
\beta }}D_{x}^{\overrightarrow{\alpha }}H\left( t,x,y\right) |_{x=y}$ as $%
t\rightarrow 0_{+}$, the ordering of derivatives within $\overrightarrow{%
\alpha }$ and $\overrightarrow{\beta }$ are not important, and we can regard 
$\overrightarrow{\alpha }$ and $\overrightarrow{\beta }$ as \emph{sets} of
their indices, respectively. This is because when we switch the ordering of
derivatives, say from $D_{x}^{i}D_{x}^{j}$ to $D_{x}^{j}D_{x}^{i}$, their
difference is a lower order differential operator, and for lower order
differential operator $D^{\overrightarrow{\alpha ^{\prime }}}$ with $%
\left\vert \overrightarrow{\alpha ^{\prime }}\right\vert <\left\vert 
\overrightarrow{\alpha }\right\vert $, $\left\vert D^{\overrightarrow{\alpha
^{\prime }}}\Phi _{t}\right\vert =o\left( \left\vert D^{\overrightarrow{%
\alpha }}\Phi _{t}\right\vert \right) $ as $t\rightarrow 0_{+}$ on $M$ by $%
\left( \ref{length}\right) $.
\end{remark}

\begin{remark}
We remark that except the constant term $A\left( \overrightarrow{\alpha },%
\overrightarrow{\beta }\right) $, the next order terms in $t^{\frac{n}{2}+%
\left[ \frac{\left\vert \overrightarrow{\alpha }\right\vert +\left\vert 
\overrightarrow{\beta }\right\vert }{2}\right] }\cdot \left. D_{y}^{%
\overrightarrow{\beta }}D_{x}^{\overrightarrow{\alpha }}H\left( t,x,y\right)
\right\vert _{x=y}$ do depend on $g$ (or curvature), and are important for
some geometric applications (e.g. Proposition \ref{Asymp-Gauss}), but to
compute them is significantly harder.
\end{remark}

\section{Applications\label{applications}}

In this section, we give several applications of the high-jet relations of $%
\psi _{t}$ in Theorem \ref{high-jet-relation} and Proposition \ref%
{high-derivative}. The applications to the mean curvature and the Riemannian
curvature tensor are related to the \emph{extrinsic geometry} and \emph{%
intrinsic geometry} of $\psi _{t}\left( M\right) \subset l^{2}$ as $%
t\rightarrow 0_{+}$ respectively. We also discuss the \emph{approximation }%
of the geometry of $\left( M,g\right) $ by finitely many eigenfunctions.

\subsection{Mean curvature of $\protect\psi _{t}\left( M\right) $ in $l^{2}$}

The first application is about the mean curvature of the embedded image $%
\psi _{t}\left( M\right) \subset l^{2}$. Intuitively speaking, as $%
t\rightarrow 0_{+}$, the extrinsic geometry of $\psi _{t}\left( M\right)
\subset l^{2}$ locally looks the same at any point, and is evenly bumpy
everywhere. More precisely we have

\begin{proposition}
\label{umbilical}For the heat kernel embedding $\psi _{t}:M\rightarrow l^{2}$
and $x\in M$, let 
\begin{equation*}
A\left( x,t\right) :T_{\psi _{t}\left( x\right) }\psi _{t}\left( M\right)
\times T_{\psi _{t}\left( x\right) }\psi _{t}\left( M\right) \rightarrow
T_{\psi _{t}\left( x\right) }\psi _{t}\left( M\right) ^{\perp }
\end{equation*}%
be the second fundamental form of the embedded image $\psi _{t}\left(
M\right) \subset l^{2}$ and $H\left( x,t\right) $ be the mean curvature
vector at $\psi _{t}\left( x\right) $. We have

\begin{enumerate}
\item (\cite{WZ} Corollary 22) $\sqrt{t}A\left( x,t\right) $ converges to
certain normal form for all $x$ on $M$ as $t\rightarrow 0_{+}$. The mean
curvature vector $H\left( x,t\right) $, after scaled by a factor $\sqrt{t}$,
converges to constant length: 
\begin{equation}
\lim_{t\rightarrow 0_{+}}\sqrt{t}\left\vert H\left( x,t\right) \right\vert =%
\sqrt{\frac{n+2}{2n}}.  \label{mean-curvature-length}
\end{equation}

\item $\psi _{t}\left( M\right) $ is asymptotically umbilical in the mean
curvature vector direction as $t\rightarrow 0_{+}$, in the following sense: 
\begin{equation*}
t\cdot \lim_{t\rightarrow 0_{+}}S\left( x,t\right) \left\langle
W,H\right\rangle =-\frac{3}{2}W\text{, for any }W\in T_{\psi _{t}\left(
x\right) }\psi _{t}\left( M\right) \text{,}
\end{equation*}%
where 
\begin{equation*}
S\left( x,t\right) :T_{\psi _{t}\left( x\right) }\psi _{t}\left( M\right)
\times T_{\psi _{t}\left( x\right) }\psi _{t}\left( M\right) ^{\perp
}\rightarrow T_{\psi _{t}\left( x\right) }\psi _{t}\left( M\right)
\end{equation*}%
is the second fundamental tensor on $\psi _{t}\left( M\right) $ defined by $%
S\left( X,\upsilon \right) =\left( \nabla _{X}\upsilon \right) ^{\intercal }$
for any $X\in T_{\psi _{t}\left( x\right) }\psi _{t}\left( M\right) $ and $%
\upsilon \in T_{\psi _{t}\left( x\right) }\psi _{t}\left( M\right) ^{\perp }$%
, $\intercal :T_{\psi _{t}\left( x\right) }l^{2}\rightarrow T_{\psi
_{t}\left( x\right) }\psi _{t}\left( M\right) $ is the projection, and
\textquotedblleft $\perp $\textquotedblright\ means the orthogonal
complement in $l^{2}$.
\end{enumerate}

The convergence is uniform for all $x$ on $M$.
\end{proposition}

\begin{proof}
For any $x\in M$, let $\left( x_{1},\cdots ,x_{n}\right) $ be a local
coordinates near $x$ such that the coordinate vectors $\left\{ \frac{%
\partial }{\partial x^{i}}\right\} _{i=1}^{n}:=\left\{ V_{i}\right\}
_{i=1}^{n}$ are orthonormal at $x$. The \emph{second fundamental form} of
the submanifold $\psi _{t}\left( M\right) \subset $ $l^{2}$ can be written as%
\begin{equation*}
A\left( x,t\right) =\Sigma _{1\leq i,j\leq n}h_{ij}\left( x,t\right)
dx^{i}dx^{j}
\end{equation*}%
where $h_{jk}\left( x,t\right) $ $\left( 1\leq j,k\leq n\right) $ are
vectors in $l^{2}$. Then as $t\rightarrow 0_{+}$, we have 
\begin{equation}
\lim_{t\rightarrow 0_{+}}\left( h_{jk}-\nabla _{V_{j}}\nabla _{V_{k}}\psi
_{t}\right) \left( x,t\right) =0\text{, for }1\leq j,k\leq n.
\label{2nd-funda-form-Taylor}
\end{equation}%
This is due to the $2$-jet relation 
\begin{eqnarray}
&&\lim_{t\rightarrow 0_{+}}\left\langle \nabla _{V_{i}}\psi _{t},\nabla
_{V_{j}}\nabla _{V_{k}}\psi _{t}\right\rangle \left( x,t\right)  \notag \\
&=&\lim_{t\rightarrow 0_{+}}\left( 4\pi t\right) ^{n/2}\left( 2t\right)
\left\langle \nabla _{V_{i}}\Phi _{t},\nabla _{V_{j}}\nabla _{V_{k}}\Phi
_{t}\right\rangle \left( x,t\right) =0  \label{2-jet-relation}
\end{eqnarray}%
from $\left( \ref{DyDxH_vanishing}\right) $, so the second order terms in
the Taylor expansion of $\psi _{t}:M\rightarrow l^{2}$ approximate the
second fundamental form of $\psi _{t}\left( M\right) \subset $ $l^{2}$ as $%
t\rightarrow 0_{+}$. Therefore, for the \emph{mean curvature} $H\left(
x,t\right) =\frac{1}{n}\Sigma _{k=1}^{n}h_{kk}\left( x\right) $, we have%
\begin{equation}
\lim_{t\rightarrow 0_{+}}\left( H\left( x,t\right) -\frac{\Sigma
_{k=1}^{n}\nabla _{V_{k}}\nabla _{V_{k}}\psi _{t}}{n}\right) \left(
x,t\right) =0.  \label{mean-curvature}
\end{equation}

Item 1 was proved in \cite{WZ} using the $2$-jet relations of $\psi _{t}$ in
Theorem \ref{uni-lin-indep}. For Item 2, using the $3$-jet relations in
Theorem \ref{high-jet-relation} for $\overrightarrow{\alpha }=\left(
i,k,k\right) $ and $\overrightarrow{\beta }=\left( j\right) $, we have%
\begin{eqnarray*}
&&\frac{\left\langle \nabla _{V_{i}}\nabla _{V_{k}}\nabla _{V_{k}}\psi
_{t},\nabla _{V_{j}}\psi _{t}\right\rangle \left( x,t\right) }{\left\vert
\nabla _{V_{i}}\nabla _{V_{k}}\nabla _{V_{k}}\psi \right\vert \left\vert
\nabla _{V_{j}}\psi _{t}\right\vert \left( x,t\right) }\rightarrow B\left( 
\overrightarrow{\alpha },\overrightarrow{\beta }\right) \\
&=&\left\{ 
\begin{tabular}{ll}
$\left( -1\right) ^{\frac{2}{2}}\cdot \frac{\left( 4!/2!\right) }{\left[
\left( 6!/3!\right) \left( 2!/1!\right) \right] ^{1/2}}=-\sqrt{\frac{3}{5}}%
\text{,}$ & $\text{ if }i=j=k,$ \\ 
$\left( -1\right) ^{\frac{2}{2}}\cdot \frac{\left( 2!/1!\right) }{\left[
\left( 2!/1!\right) \left( 2!/1!\right) \right] ^{1/2}}\cdot \frac{\left(
2!/1!\right) }{\left[ \left( 4!/2!\right) \left( 0!/0!\right) \right] ^{1/2}}%
=-\frac{1}{\sqrt{3}}\text{, }$ & $\text{\ if }i=j\neq k,$ \\ 
$0\text{,}$ & $\text{\ if }i\neq j,$%
\end{tabular}%
\right.
\end{eqnarray*}%
and%
\begin{eqnarray*}
\left\vert \nabla _{V_{i}}\nabla _{V_{k}}\nabla _{V_{k}}\psi _{t}\left(
x,t\right) \right\vert &\rightarrow &\sqrt{2}\left( 4\pi \right) ^{n/4}t^{%
\frac{n+2}{4}}\cdot \left( \frac{1}{4\pi t}\right) ^{n/4}\left( \frac{1}{2t}%
\right) ^{3/2}A\left( \overrightarrow{\alpha },\overrightarrow{\alpha }%
\right) ^{1/2} \\
&=&\left\{ 
\begin{array}{c}
\frac{1}{2t}\sqrt{\frac{6!}{2^{3}3!}}=\frac{1}{2t}\sqrt{15}\text{, if }i=k
\\ 
\frac{1}{2t}\sqrt{\frac{4!}{2^{2}2!}\cdot \frac{2!}{2^{1}2!}}=\frac{1}{2t}%
\sqrt{3}\text{, if }i\neq k%
\end{array}%
\right. , \\
\left\vert \nabla _{V_{j}}\psi _{t}\left( x,t\right) \right\vert
&\rightarrow &\sqrt{2}\left( 4\pi \right) ^{n/4}t^{\frac{n+2}{4}}\cdot
\left( \frac{1}{4\pi t}\right) ^{n/4}\left( \frac{1}{2t}\right)
^{1/2}A\left( \overrightarrow{\beta },\overrightarrow{\beta }\right)
^{1/2}=1.
\end{eqnarray*}%
Therefore as $t\rightarrow 0_{+}$, 
\begin{equation*}
\left\langle \nabla _{V_{i}}\nabla _{V_{k}}\nabla _{V_{k}}\psi _{t},\nabla
_{V_{j}}\psi _{t}\right\rangle \left( x,t\right) \rightarrow \left\{ 
\begin{array}{c}
-\sqrt{\frac{3}{5}}\cdot \frac{1}{2t}\sqrt{15}=-\frac{3}{2t}\text{, if }%
i=j=k, \\ 
-\frac{1}{\sqrt{3}}\cdot \frac{1}{2t}\sqrt{3}=-\frac{3}{2t}\text{, \ \ if }%
i=j\neq k.%
\end{array}%
\right\} =-\frac{3}{2t}.
\end{equation*}%
By the definition of the second fundamental tensor,%
\begin{eqnarray*}
&&2t\cdot \lim_{t\rightarrow 0_{+}}\left\langle S\left( x,t\right) \left(
\nabla _{V_{i}}\psi _{t},H\right) ,\nabla _{V_{j}}\psi _{t}\right\rangle
\left( x,t\right) \\
&=&2t\cdot \lim_{t\rightarrow 0_{+}}\left\langle \nabla _{\nabla
_{V_{i}}\psi _{t}}H\left( \psi _{t}\left( x\right) ,t\right) ,\nabla
_{V_{j}}\psi _{t}\right\rangle \left( \psi _{t}\left( x\right) ,t\right) \\
&=&2t\cdot \lim_{t\rightarrow 0_{+}}\left\langle \nabla _{V_{i}}H\left(
x,t\right) ,\nabla _{V_{j}}\psi _{t}\right\rangle \left( x,t\right) \\
&=&2t\cdot \lim_{t\rightarrow 0_{+}}\left\langle \nabla _{V_{i}}\left( \frac{%
1}{n}\Sigma _{k=1}^{n}\nabla _{V_{k}}\nabla _{V_{k}}\psi _{t}\right) ,\nabla
_{V_{j}}\psi _{t}\right\rangle \left( x,t\right) \\
&=&\frac{1}{n}2t\cdot \Sigma _{k=1}^{n}\lim_{t\rightarrow 0_{+}}\left\langle
\nabla _{V_{i}}\nabla _{V_{k}}\nabla _{V_{k}}\psi _{t},\nabla _{V_{j}}\psi
_{t}\right\rangle \left( x,t\right) \\
&=&\left\{ 
\begin{tabular}{ll}
$\frac{1}{n}2t\cdot \Sigma _{k=1}^{n}\left( -\frac{3}{2t}\right) ,$ & if $%
i=j,$ \\ 
$0,$ & if $i\neq j,$%
\end{tabular}%
\right. \\
&=&-3\delta _{ij}.
\end{eqnarray*}%
Since $\left\{ \nabla _{V_{j}}\psi _{t}\right\} _{1\leq j\leq n}$ span $%
T_{\psi _{t}\left( x\right) }\psi _{t}\left( M\right) $, this means for any $%
W\in T_{\psi _{t}\left( x\right) }\psi _{t}\left( M\right) $,%
\begin{equation*}
t\cdot \lim_{t\rightarrow 0_{+}}S\left( x,t\right) \left\langle
W,H\right\rangle =-\frac{3}{2}W\text{. }
\end{equation*}%
So $\psi _{t}\left( M\right) $ is asymptotically umbilical in the mean
curvature vector direction as $t\rightarrow 0_{+}$.
\end{proof}

\begin{remark}
In other words, for the image $\frac{1}{\sqrt{t}}\psi _{t}\left( M\right)
\subset l^{2}$, as $t\rightarrow 0_{+}$ its mean curvature vector converges
to constant length $\sqrt{\frac{n+2}{2n}}$, and $\frac{1}{\sqrt{t}}\psi
_{t}\left( M\right) $ is asymptotically umbilical in the mean curvature
vector direction. If we want to construct constant mean curvature
submanifolds in $\mathbb{R}^{q}$, perhaps we can truncate $\frac{1}{\sqrt{t}}%
\psi _{t}\left( M\right) \subset l^{2}$ to $\mathbb{R}^{q}$ by taking the
first $q$ components for large $q$, and then perturb it by the implicit
function theorem.
\end{remark}

\subsection{Riemannian curvature tensor}

In the second application we give a \emph{heat kernel embedding}
interpretation of the Levi-Civita connection and the Riemannian curvature
tensor. The idea is that $\psi _{t}:M\rightarrow l^{2}$ is an almost
isometric embedding, so the \textquotedblleft extrinsic\textquotedblright\
constructions of the Levi-Civita connection and Gaussian curvature of
surfaces in $\mathbb{R}^{3}$ using the ambient space can be mimicked here,
with $\mathbb{R}^{3}$ replaced by $l^{2}$. The difficulty is that we need to
take the limits as $t\rightarrow 0_{+}$ in our interpretation, and to show
these limits exist we need the high-jet relations in Proposition \ref%
{high-derivative} and some refinement.

As pointed to the author by L. I. Nicolaescu, there exists a \emph{random
function} interpretation of the Levi-Civita connection and the Riemannian
curvature tensor in Section 12 of \cite{AT} (equation (12.2.6) and Lemma
12.2.1, respectively). If the random function $f$ in \cite{AT} is taken as $%
f_{t}=\Sigma _{j\geq 1}u_{j}\phi _{j}\left( x\right) $ with the coefficients 
$u_{j}$ being independent Gaussian random variables with the expectation $%
E\left( u_{j}\right) =0$ and variance $Var\left( u_{j}\right) =e^{-\lambda
_{j}t}$, then Lemma \ref{covariant-deri} and Proposition \ref{Asymp-Gauss}
follows from the results in \cite{AT} by taking limits as $t\rightarrow 0$.
(For this approach, see \cite{Ni}, appendix B.) 

As before, for any $p\in M$, we choose the normal coordinates $\left(
x^{1},\cdots x^{n}\right) $ at $p$, and choose an orthonormal frame field $%
\left\{ V_{j}\right\} _{j=1}^{n}$ near $p$ such that $\left\{ V_{j}\right\}
_{i=1}^{n}|_{p}=\left\{ \frac{\partial }{\partial x^{j}}\right\} $. Let $%
\nabla $ be the Levi-Civita connection of $\left( M,g\right) $, and $\nabla
_{V_{j}}:=\nabla _{j}$ for $1\leq j\leq n$.

\begin{lemma}
\label{covariant-deri} (Levi-Civita Connection) For any $X,Y\in T_{p}M$, we
have 
\begin{eqnarray}
\nabla _{X}Y\left( p\right)  &=&\lim_{t\rightarrow 0_{+}}\Sigma
_{j=1}^{n}\left\langle \nabla _{X}\nabla _{Y}\psi _{t},\nabla _{V_{j}}\psi
_{t}\right\rangle \left( p\right) V_{j}\left( p\right)   \label{Levi-Civita}
\\
&=&\lim_{t\rightarrow 0_{+}}d\psi _{t}^{-1}\left[ \left( \nabla _{X}\nabla
_{Y}\psi _{t}\left( p\right) \right) ^{\intercal }\right] \text{,}  \notag
\end{eqnarray}%
where $\intercal :T_{\psi _{t}\left( p\right) }l^{2}\rightarrow T_{\psi
_{t}\left( p\right) }\psi _{t}\left( M\right) $ is the orthogonal projection
with respect to the standard metric on $l^{2}$, and on the right hand side $%
\nabla _{X}$ and $\nabla _{Y}$ are the derivatives for $l^{2}$-valued
functions $f:M\rightarrow l^{2}$.
\end{lemma}

\begin{proof}
Since all quantities in $\left( \ref{Levi-Civita}\right) $ are linear in $X$%
, we may assume $X=V_{i}$. In a neighborhood of $p$, we write $Y\left(
x\right) =\Sigma _{k=1}^{n}y^{k}\left( x\right) V_{k}\left( x\right) $.
Since $\nabla _{V_{i}}V_{j}\left( p\right) =\delta _{ij}$, we have 
\begin{equation*}
\nabla _{V_{i}}Y\left( p\right) =\Sigma _{k=1}^{n}\nabla _{i}y^{k}\left(
p\right) V_{k}\left( p\right) ,
\end{equation*}%
where $\partial _{i}y^{k}\left( x\right) =\nabla _{V_{i}}y^{k}\left(
x\right) $. On the other hand, by $\left( \ref{asymp-isom-emb}\right) $ and
the $2$-jet relation $\left( \ref{2-jet-relation}\right) $, we have%
\begin{eqnarray*}
&&\lim_{t\rightarrow 0_{+}}\Sigma _{i=1}^{n}\left\langle \nabla _{i}\nabla
_{Y}\psi _{t},\nabla _{j}\psi _{t}\right\rangle \left( p\right) V_{j}\left(
p\right) \\
&=&\lim_{t\rightarrow 0_{+}}\Sigma _{i=1}^{n}\left\langle \nabla _{i}\left(
\Sigma _{k=1}^{n}y^{k}\nabla _{k}\psi _{t}\right) ,\nabla _{j}\psi
_{t}\right\rangle \left( p\right) V_{j}\left( p\right) \\
&=&\Sigma _{i=1}^{n}\Sigma _{k=1}^{n}\lim_{t\rightarrow 0_{+}}\left( \nabla
_{i}y^{k}\left( p\right) \left\langle \nabla _{k}\psi _{t},\nabla _{j}\psi
_{t}\right\rangle \left( p\right) V_{j}\left( p\right) \right. \\
&&\left. +y^{k}\left( p\right) \left\langle \nabla _{i}\nabla _{k}\psi
_{t},\nabla _{j}\psi _{t}\right\rangle \left( p\right) V_{j}\left( p\right)
\right) \\
&=&\Sigma _{i=1}^{n}\Sigma _{k=1}^{n}\left( \partial _{i}y^{k}\left(
p\right) \delta _{kj}V_{j}\left( p\right) +0\right) \\
&=&\nabla _{i}y^{k}\left( p\right) V_{k}\left( p\right) .
\end{eqnarray*}%
Therefore the first identity%
\begin{equation*}
\nabla _{X}Y\left( p\right) =\lim_{t\rightarrow 0_{+}}\Sigma
_{j=1}^{n}\left\langle \nabla _{X}\nabla _{Y}\psi _{t},\nabla _{j}\psi
_{t}\right\rangle \left( p\right) V_{j}\left( p\right)
\end{equation*}%
is proved. Since $\left\{ V_{j}\right\} _{j=1}^{n}$ is an orthonormal basis
of $T_{p}M$, by $\left( \ref{asymp-isom-emb}\right) $ $\left\{ \nabla
_{V_{j}}\psi _{t}\left( p\right) \right\} _{j=1}^{n}$ is asymptotically an
orthonormal basis of $T_{\psi _{t}\left( p\right) }\psi _{t}\left( M\right) $
as $t\rightarrow 0_{+}$ and we have 
\begin{equation*}
\lim_{t\rightarrow 0_{+}}\left[ \left( \nabla _{X}\nabla _{Y}\psi _{t}\left(
p\right) \right) ^{\intercal }-\Sigma _{j=1}^{n}\left\langle \nabla
_{X}\nabla _{Y}\psi _{t},\nabla _{V_{j}}\psi _{t}\right\rangle \left(
p\right) \nabla _{V_{j}}\psi _{t}\left( p\right) \right] =0.
\end{equation*}%
Applying $d\psi _{t}^{-1}$ on both sides, the second identity is proved.
\end{proof}

Suppose $f:\left( M,g\right) \rightarrow \mathbb{R}^{q}$ is an isometric
embedding, and $h$ is the second fundamental form of $M$. It is well-known
that for vectors $X,Y,Z,W\in TM$, the Riemannian curvature tensor $R$ of $M$
satisfies 
\begin{equation*}
R\left( X,Y,Z,W\right) =\left\langle h\left( X,W\right) ,h\left( Y,Z\right)
\right\rangle -\left\langle h\left( X,Z\right) ,h\left( Y,W\right)
\right\rangle .
\end{equation*}%
This is the \emph{Gauss formula}, which relates the intrinsic curvature $R$
of $\left( M,g\right) $ with its extrinsic geometry in the ambient space. In
our case $\psi _{t}:M\rightarrow l^{2}$ is asymptotically isometric as $%
t\rightarrow 0_{+}$, and by $\left( \ref{2nd-funda-form-Taylor}\right) \,$%
the second fundamental form $h\left( X,Y\right) $ of $\psi _{t}\left(
M\right) \subset l^{2}$ is approximated by $\nabla _{X}\nabla _{Y}\psi _{t}$
as $t\rightarrow 0_{+}$, so it is natural to expect the following
\textquotedblleft asymptotic\textquotedblright\ Gauss formula

\begin{proposition}
\label{Asymp-Gauss} (Riemannian curvature) For any $X,Y,Z,W$ of $TM$, we
have 
\begin{equation}
R\left( X,Y,Z,W\right) =\lim_{t\rightarrow 0+}\left[ \left\langle \nabla
_{X}\nabla _{W}\psi _{t},\nabla _{Y}\nabla _{Z}\psi _{t}\right\rangle
-\left\langle \nabla _{X}\nabla _{Z}\psi _{t},\nabla _{Y}\nabla _{W}\psi
_{t}\right\rangle \right] \text{.}  \label{Gauss-formula}
\end{equation}
\end{proposition}

\begin{proof}
We only need to prove the above identity when $X,Y,Z$ and $W$ are taken from
the orthonormal basis $\left\{ V_{j}\right\} _{i=1}^{n}$ near $x$ as above,
say they are $V_{i},V_{j},V_{k}$ and $V_{l}$ respectively. For general
vector fields, we can express $X=\Sigma _{j=1}^{n}a^{j}\left( x\right)
V_{j}\left( x\right) $ etc., and plug them in $\left( \ref{Gauss-formula}%
\right) $. Then by the high-jet relations $\left( \ref{asymp-isom-emb}%
\right) $ and $\left( \ref{2-jet-relation}\right) $, as $t\rightarrow 0_{+}$
the terms of the type $\left\langle \nabla _{k}a^{i}\nabla _{i}\psi
_{t},\nabla _{l}b^{j}\nabla _{j}\psi _{t}\right\rangle $ will vanish or
cancel each other, and the terms of the type $\left\langle \nabla
_{k}a^{i}\nabla _{i}\psi _{t},\nabla _{j}\psi _{t}\nabla _{k}\psi
_{t}\right\rangle $ will converge to $0$. For the same reason, switching the
ordering of the covariant derivatives in $\left( \ref{Gauss-formula}\right) $%
, say from $\nabla _{X}\nabla _{W}\psi _{t}$ to $\nabla _{W}\nabla _{X}\psi
_{t}$ etc., will not affect the limit in $\left( \ref{Gauss-formula}\right) $%
.

We first refine our estimate of $\lim_{t\rightarrow 0_{+}}\nabla
_{V_{i}}^{x}\nabla _{V_{l}}^{x}\nabla _{V_{j}}^{y}\nabla _{V_{k}}^{y}H\left(
t,x,y\right) |_{x=y\text{ }}$in Proposition \ref{high-derivative}. We have
proved that%
\begin{equation*}
\lim_{t\rightarrow 0_{+}}\left( 4\pi t\right) ^{\frac{n}{2}}\left( 2t\right)
^{2}\left. D_{x}^{\overrightarrow{\alpha }}D_{y}^{\overrightarrow{\beta }%
}H\left( t,x,y\right) \right\vert _{x=y}=A\left( \overrightarrow{\alpha },%
\overrightarrow{\beta }\right)
\end{equation*}%
for $\overrightarrow{\alpha }=\left( i,l\right) $ and $\overrightarrow{\beta 
}=\left( j,k\right) $, where $A\left( \overrightarrow{\alpha },%
\overrightarrow{\beta }\right) $ is the leading (constant) term of $\left(
4\pi t\right) ^{\frac{n}{2}}\left( 2t\right) ^{2}\left. D_{x}^{%
\overrightarrow{\alpha }}D_{y}^{\overrightarrow{\beta }}H\left( t,x,y\right)
\right\vert _{x=y}$ as $t\rightarrow 0_{+}$, but to prove our theorem we
will need the next order term, which is linear in $t$. \bigskip Recall in
Proposition \ref{high-derivative}, for $\vec{\gamma}=\left( \overrightarrow{%
\alpha },\overrightarrow{\overline{\beta }}\right) $, (where the notation $%
\overrightarrow{\overline{\beta }}$ means the derivative $D^{\overrightarrow{%
\beta }}$ is with respect to the $y$-variables), we write $D^{%
\overrightarrow{\gamma }}H\left( t,x,y\right) =\frac{1}{\left( 4\pi t\right)
^{n/2}}e^{-\frac{r^{2}}{4t}}P_{\overrightarrow{\gamma }}\left( t,x,y\right) $%
, where $P_{\overrightarrow{\gamma }}\left( t,x,y\right) $ is a polynomial
in $D^{\overrightarrow{\mu _{s}}}\left( -\frac{r^{2}\left( x,y\right) }{4t}%
\right) $ and $\left. D^{\overrightarrow{\eta }}U\left( t,x,y\right)
\right\vert _{x=y}$, i.e. each summand is of the form%
\begin{equation*}
\left. \left( \Pi _{s=1}^{h}D^{\overrightarrow{\mu _{j}}}\left( -\frac{%
r^{2}\left( x,y\right) }{4t}\right) \right) \right\vert _{x=y}\cdot \left.
D^{\overrightarrow{\eta }}U\left( t,x,y\right) \right\vert _{x=y}\text{.}
\end{equation*}%
Using Lemma \ref{elimination-tools} and the total degree condition $%
\left\vert \vec{\gamma}\right\vert =4$, by the same method in Proposition %
\ref{high-derivative} to find the leading terms in $P_{\overrightarrow{%
\gamma }}\left( t,x,y\right) $ of order $\left( \frac{1}{t}\right) ^{2}$, we
see the secondary terms of order $\left( \frac{1}{t}\right) $ must have $h=1$%
, $\left\vert \overrightarrow{\mu _{1}}\right\vert =2$, and $\left\vert 
\overrightarrow{\eta }\right\vert =2$, or $h=1$, $\left\vert \overrightarrow{%
\mu _{1}}\right\vert =4$, and $\overrightarrow{\eta }=0$, namely of the form 
\begin{equation*}
\left. D^{\overrightarrow{\mu _{1}}}\left( -\frac{r^{2}\left( x,y\right) }{4t%
}\right) \right\vert _{x=y}\cdot \left. D^{\overrightarrow{\eta }}U\left(
t,x,y\right) \right\vert _{x=y}\text{, with }\overrightarrow{\mu _{1}}\cup 
\overrightarrow{\eta }=\overrightarrow{\gamma }\text{, }\left\vert 
\overrightarrow{\mu _{1}}\right\vert =\left\vert \overrightarrow{\eta }%
\right\vert =2,
\end{equation*}%
or 
\begin{equation*}
\left. D^{\overrightarrow{\gamma }}\left( -\frac{r^{2}\left( x,y\right) }{4t}%
\right) \right\vert _{x=y}\cdot \left. U\left( t,x,y\right) \right\vert
_{x=y}=-\frac{1}{6t}\left( R_{ijkl}\left( x\right) +R_{ilkj}\left( x\right)
\right) ,
\end{equation*}%
where we have used $\left( \ref{4th-r}\right) $. So as $t\rightarrow 0_{+}$,
we have the refined convergence 
\begin{eqnarray}
&&\left( 4\pi t\right) ^{n/2}\left( 2t\right) ^{2}\nabla _{V_{i}}^{x}\nabla
_{V_{l}}^{x}\nabla _{V_{j}}^{y}\nabla _{V_{k}}^{y}H\left( t,x,y\right)
|_{x=y}  \notag \\
&=&A\left( \overrightarrow{\left( i,l\right) },\overrightarrow{\left(
j,k\right) }\right) -t\Sigma _{\left\vert \overrightarrow{\mu _{1}}%
\right\vert =\left\vert \overrightarrow{\eta }\right\vert =2\text{, }%
\overrightarrow{\mu _{1}}\cup \overrightarrow{\eta }=\left( \overrightarrow{%
\left( i,l\right) },\overrightarrow{\left( \overline{j},\overline{k}\right) }%
\right) }\left( \left. D^{\overrightarrow{\mu _{1}}}\left( r^{2}\left(
x,y\right) \right) \right\vert _{x=y}\cdot \left. D^{\overrightarrow{\eta }%
}U\left( t,x,y\right) \right\vert _{x=y}\right)  \notag \\
&&-\frac{2}{3}t\left( R_{ijkl}\left( x\right) +R_{ilkj}\left( x\right)
\right) +O\left( t^{2}\right) .  \label{secondary-expansion}
\end{eqnarray}

For $\overrightarrow{\gamma }=\left( \overrightarrow{\left( i,l\right) },%
\overrightarrow{\left( \overline{j},\overline{k}\right) }\right) $ and $%
\overrightarrow{\gamma }=\left( \overrightarrow{\left( i,k\right) },%
\overrightarrow{\left( \overline{j},\overline{l}\right) }\right) $, by
considering all partitions of $\overrightarrow{\gamma }$ into $%
\overrightarrow{\mu _{1}}\cup \overrightarrow{\eta }$ with $\left\vert 
\overrightarrow{\mu _{1}}\right\vert =2$, and $\left\vert \overrightarrow{%
\eta }\right\vert =2$, and noticing $\nabla _{x}^{a}\nabla _{y}^{b}f\left(
x,y\right) |_{x=y}=\nabla _{y}^{a}\nabla _{x}^{b}f\left( x,y\right) |_{x=y}$
for any indices $a,b$ and functions $f$ symmetric to $x$ and $y$ (here $%
r^{2}\left( x,y\right) $ and $U\left( t,x,y\right) $), we have 
\begin{eqnarray*}
&&\Sigma _{\left\vert \overrightarrow{\mu _{1}}\right\vert =\left\vert 
\overrightarrow{\eta }\right\vert =2\text{, }\overrightarrow{\mu _{1}}\cup 
\overrightarrow{\eta }=\left( \overrightarrow{\left( i,l\right) },%
\overrightarrow{\left( \overline{j},\overline{k}\right) }\right) }\left(
\left. D^{\overrightarrow{\mu _{1}}}\left( r^{2}\left( x,y\right) \right)
\right\vert _{x=y}\cdot \left. D^{\overrightarrow{\eta }}U\left(
t,x,y\right) \right\vert _{x=y}\right)  \\
&=&\Sigma _{\left\vert \overrightarrow{\mu _{1}}\right\vert =\left\vert 
\overrightarrow{\eta }\right\vert =2\text{, }\overrightarrow{\mu _{1}}\cup 
\overrightarrow{\eta }=\left( \overrightarrow{\left( i,k\right) },%
\overrightarrow{\left( \overline{j},\overline{l}\right) }\right) }\left(
\left. D^{\overrightarrow{\mu _{1}}}\left( r^{2}\left( x,y\right) \right)
\right\vert _{x=y}\cdot \left. D^{\overrightarrow{\eta }}U\left(
t,x,y\right) \right\vert _{x=y}\right) .
\end{eqnarray*}%
We also have $A\left( \overrightarrow{\left( i,l\right) },\overrightarrow{%
\left( j,k\right) }\right) =A\left( \overrightarrow{\left( i,k\right) },%
\overrightarrow{\left( j,l\right) }\right) $ by the definition of $A\left( 
\overrightarrow{\alpha },\overrightarrow{\beta }\right) $. Therefore, from $%
\left( \ref{secondary-expansion}\right) $ we have%
\begin{eqnarray*}
&&\lim_{t\rightarrow 0_{+}}\left[ \left\langle \nabla _{V_{i}}\nabla
_{V_{l}}\psi _{t},\nabla _{V_{j}}\nabla _{V_{k}}\psi _{t}\right\rangle
\left( x\right) -\left\langle \nabla _{V_{i}}\nabla _{V_{k}}\psi _{t},\nabla
_{V_{j}}\nabla _{V_{l}}\psi _{t}\right\rangle \left( x\right) \right]  \\
&=&\lim_{t\rightarrow 0_{+}}\frac{1}{2t}\left[ \left( 4\pi t\right)
^{n/2}\left( 2t\right) ^{2}\left( \nabla _{V_{i}}^{x}\nabla
_{V_{l}}^{x}\nabla _{V_{j}}^{y}\nabla _{V_{k}}^{y}H\left( t,x,y\right)
-\nabla _{V_{i}}^{x}\nabla _{V_{k}}^{x}\nabla _{V_{j}}^{y}\nabla
_{V_{l}}^{y}H\left( t,x,y\right) \right) |_{x=y}\right]  \\
&=&\lim_{t\rightarrow 0_{+}}\frac{1}{2t}\frac{2t}{3}\left[ -R_{ijlk}\left(
x\right) -R_{iklj}\left( x\right) +R_{ijkl}\left( x\right) +R_{ilkj}\left(
x\right) +O\left( t^{2}\right) \right]  \\
&=&\frac{1}{3}\lim_{t\rightarrow 0_{+}}\left[ R_{ijkl}\left( x\right)
+R_{ikjl}\left( x\right) +R_{ijkl}\left( x\right) +R_{kjil}\left( x\right)
+O\left( t^{2}\right) \right]  \\
&=&\frac{1}{3}\lim_{t\rightarrow 0_{+}}\left[ 2R_{ijkl}\left( x\right)
+\left( R_{ikjl}\left( x\right) +R_{kjil}\left( x\right) \right) +O\left(
t^{2}\right) \right]  \\
&=&\frac{1}{3}\lim_{t\rightarrow 0_{+}}\left[ 2R_{ijkl}\left( x\right)
+R_{ijkl}\left( x\right) +O\left( t^{2}\right) \right]  \\
&=&R_{ijkl}\left( x\right) ,
\end{eqnarray*}%
where in the last four identities we have used the symmetries of the
Riemannian curvature tensor.
\end{proof}

\begin{remark}
There are several ways to express the curvature tensors by the
eigenfunctions and the heat kernel on $M$, e.g. the Ricci curvature $\left( %
\ref{u0_2nd_deri}\right) $ and scalar curvature $\left( \ref{u0u1}\right) $
in \cite{BBG}\ and \cite{Gil}, but the expression $\left( \ref{Gauss-formula}%
\right) $ has a clear geometric meaning: the Riemannian curvature tensor on $%
\left( M,g\right) $ is computed from the Gauss map of $\psi _{t}\left(
M\right) \rightarrow S^{\infty }$ in $l^{2}$ as $t\rightarrow 0_{+}$. One
more interesting point is that although our construction of the Riemannian
curvature tensor via the embedding $\psi _{t}:M\rightarrow l^{2}$ looks
\textquotedblleft extrinsic\textquotedblright , it is actually
\textquotedblleft intrinsic\textquotedblright , since the embedding $\psi
_{t}$ is constructed by the eigenfunctions of $\Delta _{g}$.
\end{remark}

If we only want to prove the \emph{existence} of the limit in $\left( \ref%
{Gauss-formula}\right) $, we only need the cancellation of the leading terms 
$A\left( \overrightarrow{\left( i,l\right) },\overrightarrow{\left(
j,k\right) }\right) $ and $A\left( \overrightarrow{\left( i,k\right) },%
\overrightarrow{\left( j,l\right) }\right) $ in $\left( \ref%
{secondary-expansion}\right) $, but do not need the symmetries of the
Riemannian curvature tensor. On the other hand, $\left( \ref{Gauss-formula}%
\right) $ \emph{manifests the symmetries} of the Riemannian curvature tensor
in a rather direct way. From $\left( \ref{Gauss-formula}\right) $, it is
easy to see $R\left( X,Y\right) Z=-R\left( Y,X\right) Z$, $\left\langle
R\left( X,Y\right) Z,W\right\rangle =-\left\langle R\left( X,Y\right)
W,Z\right\rangle $, $\left\langle R\left( X,Y\right) Z,W\right\rangle
=\left\langle R\left( Z,W\right) X,Y\right\rangle $, and $R\left( X,Y\right)
Z+R\left( Y,Z\right) X+R\left( Z,X\right) Y=0$ (the first Bianchi identity).
(Note that switching the order of differentiation in the second derivatives,
say from $\nabla _{X}\nabla _{Z}\psi _{t}$ to $\nabla _{Z}\nabla _{X}\psi
_{t}$, will not affect the limit in $\left( \ref{Gauss-formula}\right) $, as
remarked in the beginning of the proof of Proposition \ref{Asymp-Gauss}.)
With a little computation the second Bianchi identity also follows from our
formula of $R$, as follows:

\begin{lemma}
\label{Bianchi-Identity}(Second Bianchi Identity) For any coordinates $%
\left( x^{1},\cdots ,x^{n}\right) $ near $x$, the Riemannian curvature
tensor $R$ satisfies 
\begin{equation}
\frac{\partial }{\partial x^{h}}R_{klij}+\frac{\partial }{\partial x^{k}}%
R_{lhij}+\frac{\partial }{\partial x^{l}}R_{hkij}=0.  \label{Bianchi}
\end{equation}
\end{lemma}

\begin{proof}
Since all terms on the left side are tensors, we only need to prove the
identity when $\left( x^{1},\cdots ,x^{n}\right) $ is a normal coordinate
near $x$. From $\left( \ref{Gauss-formula}\right) $ we have%
\begin{eqnarray*}
\frac{\partial }{\partial x^{h}}R_{klij}\left( x\right) &=&\nabla
_{h}\lim_{t\rightarrow 0+}\left[ \left\langle \nabla _{k}\nabla _{j}\psi
_{t},\nabla _{l}\nabla _{i}\psi _{t}\right\rangle -\left\langle \nabla
_{k}\nabla _{i}\psi _{t},\nabla _{l}\nabla _{j}\psi _{t}\right\rangle \right]
\left( x\right) \\
&=&\lim_{t\rightarrow 0+}\left[ \left\langle \nabla _{hkj}\psi _{t},\nabla
_{li}\psi _{t}\right\rangle +\left\langle \nabla _{hli}\psi _{t},\nabla
_{kj}\psi _{t}\right\rangle \right. \\
&&\left. -\left\langle \nabla _{hki}\psi _{t},\nabla _{lj}\psi
_{t}\right\rangle -\left\langle \nabla _{hlj}\psi _{t},\nabla _{ki}\psi
_{t}\right\rangle \right] \left( x\right) ,
\end{eqnarray*}%
where $\nabla _{hkj}:=\nabla _{h}\nabla _{k}\nabla _{j}$. Similarly%
\begin{eqnarray*}
\frac{\partial }{\partial x^{k}}R_{lhij}\left( x\right)
&=&\lim_{t\rightarrow 0+}\left[ \left\langle \nabla _{klj}\psi _{t},\nabla
_{hi}\psi _{t}\right\rangle +\left\langle \nabla _{khi}\psi _{t},\nabla
_{lj}\psi _{t}\right\rangle \right. \\
&&\left. -\left\langle \nabla _{kli}\psi _{t},\nabla _{hj}\psi
_{t}\right\rangle -\left\langle \nabla _{khj}\psi _{t},\nabla _{li}\psi
_{t}\right\rangle \right] \left( x\right) , \\
\frac{\partial }{\partial x^{l}}R_{hkij}\left( x\right)
&=&\lim_{t\rightarrow 0+}\left[ \left\langle \nabla _{lhj}\psi _{t},\nabla
_{ki}\psi _{t}\right\rangle +\left\langle \nabla _{lki}\psi _{t},\nabla
_{hj}\psi _{t}\right\rangle \right. \\
&&\left. -\left\langle \nabla _{lhi}\psi _{t},\nabla _{kj}\psi
_{t}\right\rangle -\left\langle \nabla _{lkj}\psi _{t},\nabla _{hi}\psi
_{t}\right\rangle \right] \left( x\right) .
\end{eqnarray*}%
Adding the $3$ identities together, $\left( \ref{Bianchi}\right) $ follows.
There is a convergence issue of limits like $\lim_{t\rightarrow
0+}\left\langle \nabla _{hkj}\psi _{t},\nabla _{li}\psi _{t}\right\rangle
=O\left( \frac{1}{t}\right) $ from Proposition \ref{high-derivative}, but
writing%
\begin{equation*}
\frac{\partial }{\partial x^{h}}R_{klij}+\frac{\partial }{\partial x^{k}}%
R_{lhij}+\frac{\partial }{\partial x^{l}}R_{hkij}=\frac{1}{2t}\left[
2t\left( \frac{\partial }{\partial x^{h}}R_{klij}+\frac{\partial }{\partial
x^{k}}R_{lhij}+\frac{\partial }{\partial x^{l}}R_{hkij}\right) \right]
\end{equation*}%
and applying Proposition \ref{high-derivative} easily fix that.
\end{proof}

\subsection{Approximation of Riemannian geometry by eigenfunctions}

For the embedding map $\psi _{t}:\left( M,g\right) \rightarrow l^{2}$, we
can truncate the infinite dimensional $l^{2}$ to $\mathbb{R}^{q}$ by taking
its first $q$ components, and denote the resulted map by $\psi _{t}^{q}$.
The following truncation lemma in \cite{WZ} gives the bound of $q$ that can
preserve most good properties of $\psi _{t}$. It was proved by the Weyl's
asymptotic formula of eigenvalues (see \cite{Ch}) and $C^{0}$-estimates of
eigenfunctions (and their derivatives) on compact Riemannian manifolds.

\begin{lemma}
\label{truncation}(\cite{WZ} Theorem 6 and Remark 7)\label{isom-truncation}
Let $\rho >0$ be a fixed small constant. For integers $q=q\left( t\right) $
of the order $t^{-\left( \frac{n}{2}+\rho \right) }$ as $t\rightarrow 0_{+}$%
, the truncated heat kernel mapping 
\begin{equation*}
\psi _{t}^{q\left( t\right) }:\left( M,g\right) \rightarrow \mathbb{R}%
^{q\left( t\right) }
\end{equation*}%
still satisfies the asymptote%
\begin{equation}
\left( \psi _{t}^{q\left( t\right) }\right) ^{\ast }g_{0}=g+\frac{t}{3}%
\left( \frac{1}{2}S_{g}\cdot g-\text{Ric}_{g}\right) +O\left( t^{2}\right) ,
\label{truncate-isom-emb}
\end{equation}%
where $g_{0}$ is the standard metric in $l^{2}$. The truncation order $%
q\left( t\right) \sim t^{-\left( \frac{n}{2}+\rho \right) }$ is also valid
to preserve all high-jet relations in Theorem \ref{high-jet-relation} for
higher derivatives of $\psi _{t}$.
\end{lemma}

Using the above truncation lemma, we can approximate the metric $g$, the
Levi-Civita connection $\nabla $, and the Riemannian curvature $R$ by $%
q\left( t\right) \sim t^{-\frac{n}{2}-\rho }$ eigenfunctions by the
expressions $\left( \ref{asymp-isom-emb}\right) $, $\left( \ref{Levi-Civita}%
\right) $, and $\left( \ref{Gauss-formula}\right) $, respectively, with the
error of certain orders of $t$. This suggests the possibility of
approximating the Riemannian geometry of $\left( M,g\right) $ by finitely
many eigenfunctions. Indeed, if the Schwartz function $w$ is chosen with
compact support, Theorem 1.6 of \cite{Ni} gives one of such approximations. 

It will be useful to get a precise control of the coefficients in the above
remainder $O\left( t^{2}\right) $ and their higher derivative counterparts,
so we know how large $q\left( t\right) $ (i.e. how small $t$) should be in
order to yield a given accuracy.

\section{Algebraic structures in the $\infty $-jet space of $\Phi _{t}\left(
x\right) \label{reformulations}$}

In Section \ref{high-jet relations} we have obtained the angle and length
relations of all derivative vectors $D^{\overrightarrow{\alpha }}\Phi
_{t}\left( x\right) $ as $t\rightarrow 0_{+}$. These relations are
independent on $g$, $x$ and the choice of orthonormal basis, so it is
worthwhile to investigate them more closely. The relations are expressed by
the constants $A\left( \overrightarrow{\alpha },\overrightarrow{\beta }%
\right) $ and $B\left( \overrightarrow{\alpha },\overrightarrow{\beta }%
\right) $. In this section we explore more relations between these
constants, or in other words study the $\infty $-jet space of $\Phi
_{t}\left( x\right) $ as $t\rightarrow 0_{+}$.

\subsection{\protect\bigskip Inductive relations of $A\left( \protect%
\overrightarrow{\protect\alpha },\protect\overrightarrow{\protect\beta }%
\right) $ and $B\left( \protect\overrightarrow{\protect\alpha },\protect%
\overrightarrow{\protect\beta }\right) $}

It turns out the constants $A\left( \overrightarrow{\alpha },\overrightarrow{%
\beta }\right) $ and $B\left( \overrightarrow{\alpha },\overrightarrow{\beta 
}\right) $ involved in the high-jet relations in Theorem \ref%
{high-jet-relation} can be defined inductively. Suppose $j_{r}$ $\in \left\{
1,2,\cdots n\right\} $, but not necessarily in $\mathcal{S=}\left\{ 
\overrightarrow{\alpha }\right\} \amalg \left\{ \overrightarrow{\beta }%
\right\} $. Let $\overrightarrow{\alpha }_{+}=\left( \overrightarrow{\alpha }%
,j_{r}\right) $, $\overrightarrow{\beta }_{+}=\left( \overrightarrow{\beta }%
,j_{r}\right) $, and $\overrightarrow{\alpha _{+k}}=\left( \overrightarrow{%
\alpha },\underset{k}{\underbrace{j_{r},\cdots ,j_{r}}}\right) $. If $%
j_{r}\in \overrightarrow{\beta }$, we let $\overrightarrow{\beta }_{-}=%
\overrightarrow{\beta }\backslash \left\{ j_{r}\right\} $. By definition it
is easy to check $A\left( \overrightarrow{\alpha },\overrightarrow{\beta }%
\right) $ satisfies the following properties: 
\begin{eqnarray}
\text{Normalization}\text{: } &&A\left( \overrightarrow{\varnothing },%
\overrightarrow{\varnothing }\right) =1,\text{ (where }\varnothing \text{ is
the empty set)}  \label{unit} \\
\text{Symmetry}\text{: } &&A\left( \overrightarrow{\alpha },\overrightarrow{%
\beta }\right) =A\left( \overrightarrow{\beta },\overrightarrow{\alpha }%
\right) ,  \label{symm} \\
\text{Leibniz rule}\text{: } &&A\left( \overrightarrow{\alpha }_{+},%
\overrightarrow{\beta }_{-}\right) =-A\left( \overrightarrow{\alpha },%
\overrightarrow{\beta }\right) ,  \label{switch} \\
\text{Adding index}\text{: } &&A\left( \overrightarrow{\alpha }_{+},%
\overrightarrow{\beta }_{+}\right) =A\left( \overrightarrow{\alpha },%
\overrightarrow{\beta }\right) \left( a_{r}+b_{r}+1\right) .  \label{adding}
\end{eqnarray}%
These properties completely characterize $A\left( \overrightarrow{\alpha },%
\overrightarrow{\beta }\right) $. Indeed, to obtain a general $A\left( 
\overrightarrow{\alpha },\overrightarrow{\beta }\right) $, we can start from 
$\overrightarrow{\alpha }=\overrightarrow{\beta }=\overrightarrow{%
\varnothing }$, use $\left( \ref{adding}\right) $ to obtain $A\left( 
\overrightarrow{\varnothing _{+k}},\overrightarrow{\varnothing _{+k}}\right) 
$, and then use $\left( \ref{switch}\right) $ to obtain $A\left( 
\overrightarrow{\varnothing _{+\left( k+l\right) }},\overrightarrow{%
\varnothing _{+\left( k-l\right) }}\right) $ for $0\leq l\leq k$. Then we
can use $\left( \text{\ref{adding}}\right) $ to add different indices $%
j^{\prime }$ to $\left( \overrightarrow{\alpha },\overrightarrow{\beta }%
\right) $ to get all $A\left( \overrightarrow{\alpha },\overrightarrow{\beta 
}\right) $ (note if $j^{\prime }\notin \left\{ \overrightarrow{\alpha }%
\right\} \amalg \left\{ \overrightarrow{\beta }\right\} $, then $A\left( 
\overrightarrow{\alpha _{+}},\overrightarrow{\beta _{+}}\right) =A\left( 
\overrightarrow{\alpha },\overrightarrow{\beta }\right) $ by $\left( \text{%
\ref{adding}}\right) $). Similarly, $B\left( \overrightarrow{\alpha },%
\overrightarrow{\beta }\right) $ is characterized by%
\begin{eqnarray}
B\left( \overrightarrow{\varnothing },\overrightarrow{\varnothing }\right)
&=&1,\text{ }B\left( \overrightarrow{\alpha },\overrightarrow{\beta }\right)
=B\left( \overrightarrow{\beta },\overrightarrow{\alpha }\right) ,  \notag \\
B\left( \overrightarrow{\alpha }_{+},\overrightarrow{\beta }_{-}\right)
&=&-B\left( \overrightarrow{\alpha },\overrightarrow{\beta }\right) \left( 
\frac{2b_{r}-1}{2a_{r}+1}\right) ^{1/2},  \notag \\
B\left( \overrightarrow{\alpha }_{+},\overrightarrow{\beta }_{+}\right)
&=&B\left( \overrightarrow{\alpha },\overrightarrow{\beta }\right) \frac{%
a_{r}+b_{r}+1}{\left[ \left( 2a_{r}+1\right) \left( 2b_{r}+1\right) \right]
^{1/2}}.  \label{B-adding-index}
\end{eqnarray}

\subsection{Stabilization property from repeated differentials\label%
{stablization}}

For some index $j_{r}$ $\in \left\{ 1,2,\cdots n\right\} $ (but not
necessarily in $\left\{ \overrightarrow{\alpha }\right\} \amalg \left\{ 
\overrightarrow{\beta }\right\} $), if we apply $\partial _{j_{r}}$
repeatedly to $D^{\overrightarrow{\alpha }}\Phi _{t}\left( x\right) $ and $%
D^{\overrightarrow{\beta }}\Phi _{t}\left( x\right) $, then the the resulted
derivative vectors become \textquotedblleft more and more
collinear\textquotedblright , and the angle between them has certain
stabilization property. More precisely, we have the following

\begin{proposition}
\label{repeated-diffrential}$B\left( \overrightarrow{\alpha }_{+},%
\overrightarrow{\beta }_{+}\right) $ and $B\left( \overrightarrow{\alpha },%
\overrightarrow{\beta }\right) $ always have the same sign, and%
\begin{equation}
\left\vert B\left( \overrightarrow{\alpha }_{+},\overrightarrow{\beta }%
_{+}\right) \right\vert \geq \left\vert B\left( \overrightarrow{\alpha },%
\overrightarrow{\beta }\right) \right\vert \,,  \label{monotone-angle}
\end{equation}%
where the equality holds if and only if $j_{r}$ appears with the same
multiplicity in $\overrightarrow{\alpha }$ and $\overrightarrow{\beta }$.
Further more we have the following limit: if we let the multi-index $%
\overrightarrow{\alpha _{+k}}=\left( \overrightarrow{\alpha },\underset{k}{%
\underbrace{j_{r},\cdots ,j_{r}}}\right) $, similarly for $\overrightarrow{%
\beta _{+k}}$, then 
\begin{equation}
\lim_{k\rightarrow \infty }B\left( \overrightarrow{\alpha _{+k}},%
\overrightarrow{\beta _{+k}}\right) :=B\left( \overrightarrow{\alpha }_{\ast
},\overrightarrow{\beta }_{\ast }\right) \,,  \label{B_jr_ab}
\end{equation}%
where the multi-indices $\overrightarrow{\alpha _{\ast }}$ and $%
\overrightarrow{\beta }_{\ast }$ are obtained from deleting all $j_{r}$
indices in $\overrightarrow{\alpha }$ and $\overrightarrow{\beta }$
respectively. We also have%
\begin{equation}
\lim_{k\rightarrow \infty }B\left( \overrightarrow{\alpha },\overrightarrow{%
\beta _{+k}}\right) :=0.  \label{B_jr_b}
\end{equation}
\end{proposition}

\begin{proof}
Since $\left\vert \overrightarrow{\alpha }_{+}\right\vert -\left\vert 
\overrightarrow{\beta }_{+}\right\vert =\left\vert \overrightarrow{\alpha }%
\right\vert -\left\vert \overrightarrow{\beta }\right\vert $, by definition $%
B\left( \overrightarrow{\alpha }_{+},\overrightarrow{\beta }_{+}\right) $
and $B\left( \overrightarrow{\alpha },\overrightarrow{\beta }\right) $
always have the same sign. Let us assume $B\left( \overrightarrow{\alpha },%
\overrightarrow{\beta }\right) \geq 0$ (the $B\left( \overrightarrow{\alpha }%
,\overrightarrow{\beta }\right) \leq 0$ case is similar). From $\left( \ref%
{B-adding-index}\right) $ we have%
\begin{eqnarray*}
B\left( \overrightarrow{\alpha }_{+},\overrightarrow{\beta }_{+}\right)
&=&B\left( \overrightarrow{\alpha },\overrightarrow{\beta }\right) \cdot
\left( \frac{\left( a_{r}+b_{r}+1\right) ^{2}}{\left( 2a_{r}+1\right) \left(
2b_{r}+1\right) }\right) ^{1/2} \\
&=&B\left( \overrightarrow{\alpha },\overrightarrow{\beta }\right) \cdot
\left( 1+\frac{\left( a_{r}-b_{r}\right) ^{2}}{\left( 2a_{r}+1\right) \left(
2b_{r}+1\right) }\right) ^{1/2} \\
&\geq &B\left( \overrightarrow{\alpha },\overrightarrow{\beta }\right) \text{%
,}
\end{eqnarray*}%
and the equality holds if and only if $a_{r}=b_{r}$. Recall 
\begin{equation*}
B\left( \overrightarrow{\alpha },\overrightarrow{\beta }\right) =\left(
-1\right) ^{\frac{\left\vert \overrightarrow{\alpha }\right\vert -\left\vert 
\overrightarrow{\beta }\right\vert }{2}}{\LARGE \Pi }_{r=1}^{s}\left. \frac{%
\left( 2\sigma _{r}\right) !}{\sigma _{r}!}\right/ \left[ \frac{\left(
2a_{r}\right) !}{a_{r}!}\frac{\left( 2b_{r}\right) !}{b_{r}!}\right] ^{1/2}.
\end{equation*}%
Since $\sigma _{r}=\frac{a_{r}+b_{r}}{2}$, if we let $l_{r}=\frac{b_{r}-a_{r}%
}{2}\geq 0$ (assuming $b_{r}\geq a_{r}$), then%
\begin{equation*}
\frac{a_{r}!}{\sigma _{r}!}\frac{b_{r}!}{\sigma _{r}!}=\frac{\left( \sigma
_{r}+1\right) \cdots \left( \sigma _{r}+l_{r}\right) }{\left( \sigma
_{r}-1\right) \cdots \left( \sigma _{r}-l_{r}\right) },\text{ and }\frac{%
\left( 2a_{r}\right) !}{\left( 2\sigma _{r}\right) !}\frac{\left(
2b_{r}\right) !}{\left( 2\sigma _{r}\right) !}=\frac{\left( 2\sigma
_{r}+1\right) \cdots \left( 2\sigma _{r}+2l_{r}\right) }{\left( 2\sigma
_{r}-1\right) \cdots \left( 2\sigma _{r}-2l_{r}\right) }.\text{ }
\end{equation*}%
So the $r$-th factor in $B\left( \overrightarrow{\alpha },\overrightarrow{%
\beta }\right) $ is%
\begin{eqnarray}
&&\left. \frac{\left( 2\sigma _{r}\right) !}{\sigma _{r}!}\right/ \left[ 
\frac{\left( 2a_{r}\right) !}{a_{r}!}\frac{\left( 2b_{r}\right) !}{b_{r}!}%
\right] ^{1/2}=\left[ \frac{\left( 2\sigma _{r}\right) !}{\left(
2a_{r}\right) !}\frac{\left( 2\sigma _{r}\right) !}{\left( 2b_{r}\right) !}%
\frac{a_{r}!}{\sigma _{r}!}\frac{b_{r}!}{\sigma _{r}!}\right] ^{1/2}  \notag
\\
&=&\left[ \frac{\left( 2\sigma _{r}-1\right) \left( 2\sigma _{r}-3\right)
\cdots \left( 2\sigma _{r}-2l_{r}+1\right) }{\left( 2\sigma _{r}+1\right)
\left( 2\sigma _{r}+3\right) \cdots \left( 2\sigma _{r}+2l_{r}-1\right) }%
\right] ^{1/2}\leq 1.  \label{r-factor-clear}
\end{eqnarray}%
As $k\rightarrow \infty $, in the multi-index $\left( \overrightarrow{\alpha
_{+k}},\overrightarrow{\beta _{+k}}\right) $, $\sigma _{r}$ is changed to $%
\sigma _{r}+k$, but $l_{r}$ is unchanged, so the factor involving $j_{r}$ in 
$B\left( \overrightarrow{\alpha _{+k}},\overrightarrow{\beta _{+k}}\right) $
becomes%
\begin{equation*}
\left[ \frac{\left( 2\sigma _{r}+2k-1\right) \left( 2\sigma _{r}+2k-3\right)
\cdots \left( 2\sigma _{r}+2k-2l_{r}+1\right) }{\left( 2\sigma
_{r}+2k+1\right) \left( 2\sigma _{r}+2k+3\right) \cdots \left( 2\sigma
_{r}+2k+2l_{r}-1\right) }\right] ^{1/2}\rightarrow 1
\end{equation*}%
as $k\rightarrow \infty $, and the other factors not involving $j_{r}$ are
unchanged. This gives the limit $B\left( \overrightarrow{\alpha _{\ast }},%
\overrightarrow{\beta _{\ast }}\right) $.

For $\left( \ref{B_jr_b}\right) $, w.l.o.g we can assume $\left\vert 
\overrightarrow{\alpha }\right\vert +\left\vert \overrightarrow{\beta }%
\right\vert $ is even (otherwise we can replace $\overrightarrow{\beta }$ by 
$\overrightarrow{\beta }_{+}$ before we take limit). In the multi-index $%
\left( \overrightarrow{\alpha },\overrightarrow{\beta _{+2k}}\right) $, $%
\sigma _{r}$ and $l_{r}$ are changed to $\sigma _{r}+k$ and $l_{r}+k$
respectively, so the factor involving $j_{r}$ in $B\left( \overrightarrow{%
\alpha },\overrightarrow{\beta _{+2k}}\right) $ becomes%
\begin{equation*}
\left[ \frac{\left( 2\sigma _{r}+2k-1\right) \left( 2\sigma _{r}+2k-3\right)
\cdots \left( 2\sigma _{r}-2l_{r}+1\right) }{\left( 2\sigma _{r}+2k+1\right)
\left( 2\sigma _{r}+2k+3\right) \cdots \left( 2\sigma
_{r}+4k+2l_{r}-1\right) }\right] ^{1/2}\leq \left[ \frac{2\sigma
_{r}-2l_{r}+1}{2\sigma _{r}+4k+2l_{r}-1}\right] ^{1/2}\rightarrow 0
\end{equation*}%
as $k\rightarrow \infty $. Therefore $B\left( \overrightarrow{\alpha },%
\overrightarrow{\beta _{+2k}}\right) \rightarrow 0$. By our definition $%
B\left( \overrightarrow{\alpha },\overrightarrow{\beta _{+\left( 2k+1\right)
}}\right) =0$. Combining the two cases, $\lim_{k\rightarrow \infty }B\left( 
\overrightarrow{\alpha },\overrightarrow{\beta _{+k}}\right) =0$.
\end{proof}

\begin{corollary}
\label{non-degn} $-1<B\left( \overrightarrow{\alpha },\overrightarrow{\beta }%
\right) \leq 1$, and $B\left( \overrightarrow{\alpha },\overrightarrow{\beta 
}\right) =1$ if and only if $\overrightarrow{\alpha }=\overrightarrow{\beta }
$.
\end{corollary}

\begin{proof}
From $\left( \ref{r-factor-clear}\right) $ we see for any multi-indices $%
\overrightarrow{\alpha }$ and $\overrightarrow{\beta }$, 
\begin{equation}
\left\vert B\left( \overrightarrow{\alpha },\overrightarrow{\beta }\right)
\right\vert \leq 1.  \label{1-bound}
\end{equation}%
If $B\left( \overrightarrow{\alpha },\overrightarrow{\beta }\right) =1$ but $%
\overrightarrow{\alpha }\neq \overrightarrow{\beta }$, then there must be
some index $j_{r\text{ }}$ appearing with \emph{different} multiplicities in 
$\overrightarrow{\alpha }$ and $\overrightarrow{\beta }$. We can construct $%
\overrightarrow{\alpha _{+}}$ and $\overrightarrow{\beta _{+}}$ by adding
the index $j_{r}$. By $\left( \ref{monotone-angle}\right) $ (notice the
equality can \emph{not} hold by the condition on $j_{r}$), we have%
\begin{equation*}
\left\vert B\left( \overrightarrow{\alpha }_{+},\overrightarrow{\beta }%
_{+}\right) \right\vert >\left\vert B\left( \overrightarrow{\alpha },%
\overrightarrow{\beta }\right) \right\vert =1,
\end{equation*}%
contradicting with $\left( \ref{1-bound}\right) $. Therefore $%
\overrightarrow{\alpha }=\overrightarrow{\beta }$.

We claim $B\left( \overrightarrow{\alpha },\overrightarrow{\beta }\right)
\neq -1$. Otherwise, by the sign of $B\left( \overrightarrow{\alpha },%
\overrightarrow{\beta }\right) $, we see $\frac{\left\vert \overrightarrow{%
\alpha }\right\vert -\left\vert \overrightarrow{\beta }\right\vert }{2}$
must be odd, especially $\overrightarrow{\alpha }\neq \overrightarrow{\beta }
$. Therefore there is some $l_{r}\neq 0$. By $\left( \ref{r-factor-clear}%
\right) $, this makes the $r$-th factor in $B\left( \overrightarrow{\alpha },%
\overrightarrow{\beta }\right) $ have absolute value strictly less than $1$,
so $\left\vert B\left( \overrightarrow{\alpha },\overrightarrow{\beta }%
\right) \right\vert <1$. But this contradicts with $\left\vert B\left( 
\overrightarrow{\alpha },\overrightarrow{\beta }\right) \right\vert
=\left\vert -1\right\vert =1$.
\end{proof}

\subsection{Lattice geometry on $\mathbb{Z}_{+}^{n}$}

In this subsection, we use the high-jet relations of $\Phi _{t}$ in Theorem %
\ref{high-jet-relation} to put a metric $d$ on the lattice space 
\begin{equation*}
\mathbb{Z}_{+}^{n}=\left\{ \left( m_{1},\cdots ,m_{n}\right) |m_{j}\in 
\mathbb{Z}\text{, }m_{j}\geq 0\text{ for }1\leq j\leq n\right\} .
\end{equation*}%
The motivation is to visualize these relations by the geometry of $\left( 
\mathbb{Z}_{+}^{n},d\right) $.

For any multi-index $\overrightarrow{\alpha }$ with indices in $\left\{
1,2,\cdots ,n\right\} $, taking the multiplicity of its each index, we
obtain a point $z\in $ $\mathbb{Z}_{+}^{n}$; conversely we can use $z$ to
recover $\overrightarrow{\alpha }$. For example, if we let $%
z_{1},z_{2},z_{3} $ be the three axes of $\mathbb{Z}_{+}^{3}$, then the
multi-index $\overrightarrow{\alpha }=\left( 1,1,2,3\right) $ has $1$ with
multiplicity $2 $, $2$ with multiplicity $1$, and $3$ with multiplicity $1$.
So $\overrightarrow{\alpha }$ correspond to $z=\left( 2,1,1\right) \in 
\mathbb{Z}_{+}^{3}$. \ With this identification, we have

\begin{definition}
\label{distance-ab}(Angle distance) For any two multi-indices $%
\overrightarrow{\alpha }$ and $\overrightarrow{\beta }$ in $\mathbb{Z}%
_{+}^{n}$, we define their \emph{angle distance} to be%
\begin{equation}
d\left( \overrightarrow{\alpha },\overrightarrow{\beta }\right) =\cos
^{-1}\left( B\left( \overrightarrow{\alpha },\overrightarrow{\beta }\right)
\right) .  \label{Zn-distance}
\end{equation}
\end{definition}

By Theorem \ref{high-jet-relation}, $d\left( \overrightarrow{\alpha },%
\overrightarrow{\beta }\right) $ is the \emph{limiting distance} of $\frac{%
D^{\overrightarrow{\alpha }}\Phi _{t}\left( x\right) }{\left\vert D^{%
\overrightarrow{\alpha }}\Phi _{t}\left( x\right) \right\vert }$ and $\frac{%
D^{\overrightarrow{\beta }}\Phi _{t}\left( x\right) }{\left\vert D^{%
\overrightarrow{\beta }}\Phi _{t}\left( x\right) \right\vert }$ on the unit
sphere $S^{\infty }\subset l^{2}$ as $t\rightarrow 0_{+}$. Let $\theta
_{t}\left( \overrightarrow{\alpha },\overrightarrow{\beta }\right) $ be the
minimal distance of $\frac{D^{\overrightarrow{\alpha }}\Phi _{t}\left(
x\right) }{\left\vert D^{\overrightarrow{\alpha }}\Phi _{t}\left( x\right)
\right\vert }$ and $\frac{D^{\overrightarrow{\beta }}\Phi _{t}\left(
x\right) }{\left\vert D^{\overrightarrow{\beta }}\Phi _{t}\left( x\right)
\right\vert }$ on $S^{\infty }$, then $0\leq \theta _{t}\left( 
\overrightarrow{\alpha },\overrightarrow{\beta }\right) \leq \pi $, and 
\begin{equation}
\lim_{t\rightarrow 0_{+}}\theta _{t}\left( \overrightarrow{\alpha },%
\overrightarrow{\beta }\right) =d\left( \overrightarrow{\alpha },%
\overrightarrow{\beta }\right) .  \label{sphere-distance}
\end{equation}

\begin{lemma}
The angle distance $d$ defined in $\left( \ref{Zn-distance}\right) $ is a
metric on $\mathbb{Z}_{+}^{n}$. For any $\overrightarrow{\alpha }$ and $%
\overrightarrow{\beta }\in \mathbb{Z}_{+}^{n}$, $0\leq d\left( 
\overrightarrow{\alpha },\overrightarrow{\beta }\right) <\pi $.
\end{lemma}

\begin{proof}
To prove $d$ is a metric, we only need to verify that it is \emph{%
nondegenerate}, since the triangle inequality of $\theta _{t}\left( 
\overrightarrow{\alpha },\overrightarrow{\beta }\right) $ on $S^{\infty }$
is preserved when we take the limit in $\left( \ref{sphere-distance}\right) $
as $t\rightarrow 0_{+}$, and $d\left( \overrightarrow{\alpha },%
\overrightarrow{\beta }\right) =d\left( \overrightarrow{\beta },%
\overrightarrow{\alpha }\right) $ follows from the definition of $B\left( 
\overrightarrow{\alpha },\overrightarrow{\beta }\right) $. Suppose $d\left( 
\overrightarrow{\alpha },\overrightarrow{\beta }\right) =0$, then $B\left( 
\overrightarrow{\alpha },\overrightarrow{\beta }\right) =1$. From Corollary %
\ref{non-degn}, we have $\overrightarrow{\alpha }=\overrightarrow{\beta }$.
By Corollary \ref{non-degn}, $B\left( \overrightarrow{\alpha },%
\overrightarrow{\beta }\right) \in (-1,1]$, so $0\leq d\left( 
\overrightarrow{\alpha },\overrightarrow{\beta }\right) <\pi $.
\end{proof}

We can interpret the stabilization property in $\left( \ref{B_jr_ab}\right) $
of Proposition \ref{repeated-diffrential} in terms of the distance $d$ on $%
\mathbb{Z}_{+}^{n}$:

\begin{lemma}
\label{Stable}(Stabilization) For any coordinate vector $v=\frac{\partial }{%
\partial z^{j}}$, and any two vectors $\overrightarrow{\alpha },%
\overrightarrow{\beta }\in \mathbb{Z}_{+}^{n}$, we have 
\begin{eqnarray*}
\left\vert d\left( \overrightarrow{\alpha }+v,\overrightarrow{\beta }%
+v\right) -\frac{\pi }{2}\right\vert &\geq &\left\vert d\left( 
\overrightarrow{\alpha },\overrightarrow{\beta }\right) -\frac{\pi }{2}%
\right\vert , \\
\lim_{k\rightarrow \infty }d\left( \overrightarrow{\alpha }+kv,%
\overrightarrow{\beta }+kv\right) &=&d\left( \overrightarrow{\alpha }_{\ast
},\overrightarrow{\beta }_{\ast }\right) ,\text{ } \\
\lim_{k\rightarrow \infty }d\left( \overrightarrow{\alpha },\overrightarrow{%
\beta }+kv\right) &=&\frac{\pi }{2},
\end{eqnarray*}%
where $\overrightarrow{\alpha }_{\ast }$ and $\overrightarrow{\beta }_{\ast
} $ are obtained by setting the $j$-th component of $\overrightarrow{\alpha }
$ and $\overrightarrow{\beta }$ to be zero.
\end{lemma}

The stabilization properties in the above lemma make each axis of $\left( 
\mathbb{Z}_{+}^{n},d\right) $ somehow looks like a real projective line. It
will be interesting to say more about the geometry of $\left( \mathbb{Z}%
_{+}^{n},d\right) $ as $\left\vert \overrightarrow{\alpha }\right\vert
\rightarrow \infty $.

In the following we are interested in the orthogonal relations of $D^{%
\overrightarrow{\alpha }}\Phi _{t}\left( x\right) $ as $t\rightarrow 0_{+}$.
From our definition, $B\left( \overrightarrow{\alpha },\overrightarrow{\beta 
}\right) =0$ if and only if some index appears with odd multiplicity in $%
\left\{ \overrightarrow{\alpha }\right\} \amalg \left\{ \overrightarrow{%
\beta }\right\} $. This can be reformulated as

\begin{lemma}
(Orthogonal relations) $d\left( \overrightarrow{\alpha },\overrightarrow{%
\beta }\right) =\frac{\pi }{2}$ if and only if $\overrightarrow{\alpha }-%
\overrightarrow{\beta }\notin \left( 2\mathbb{Z}\right) ^{n}$. Consequently,

\begin{enumerate}
\item Any two adjacent lattice points in $\mathbb{Z}_{+}^{n}$ have the angle
distance $\frac{\pi }{2}$;

\item $\mathbb{Z}_{+}^{n}/\left( 2\mathbb{Z}_{+}\right) ^{n}$ has $2^{n}$
\textquotedblleft orthogonal\textquotedblright\ cosets $A_{1},\cdots
,A_{2^{n}}\subset \mathbb{Z}_{+}^{n}$, such that for any two elements in
different cosets, their angle distance is $\frac{\pi }{2}$.
\end{enumerate}
\end{lemma}

This $\mathbb{Z}_{+}^{n}/\left( 2\mathbb{Z}_{+}\right) ^{n}$-grading of the $%
\infty $-jet space of $\Phi _{t}$ as $t\rightarrow 0$ may have future
applications.

In the following we give some estimate of $d\left( \overrightarrow{\alpha },%
\overrightarrow{\beta }\right) $. Let $d_{0}$ be the \textquotedblleft \emph{%
set distance}\textquotedblright\ function on $\mathbb{Z}_{+}^{n}$: for $%
\overrightarrow{\alpha }=\left( a_{1},\cdots ,a_{n}\right) $ and $%
\overrightarrow{\beta }=\left( b_{1},\cdots ,b_{n}\right) $,%
\begin{equation*}
d_{0}\left( \overrightarrow{\alpha },\overrightarrow{\beta }\right) =\Sigma
_{j=1}^{n}\left\vert a_{j}-b_{j}\right\vert =\left\vert \overrightarrow{%
\alpha }\bigtriangleup \overrightarrow{\beta }\right\vert ,
\end{equation*}%
where $\bigtriangleup $ is the difference of two sets $A$ and $B$: $%
A\bigtriangleup B=\left( A\backslash B\right) \cup \left( B\backslash
A\right) $.

Intuitively, if $d_{0}\left( \overrightarrow{\alpha },\overrightarrow{\beta }%
\right) $ is very large relative to $\left\vert \overrightarrow{\alpha }%
\right\vert +\left\vert \overrightarrow{\beta }\right\vert $, then $%
\overrightarrow{\alpha }$ and $\overrightarrow{\beta }$ are very different,
and we expect the angle $d\left( \overrightarrow{\alpha },\overrightarrow{%
\beta }\right) $ between the derivative vectors $D^{\overrightarrow{\alpha }%
}\Phi _{t}\left( x\right) $ and $D^{\overrightarrow{\beta }}\Phi _{t}\left(
x\right) $ to be not too close to $0$ or $\pi $ as $t\rightarrow 0_{+}$.
Indeed, we have

\begin{proposition}
(Distance comparison) There exists a constant $\delta >0$, such that 
\begin{equation}
\left\vert \cos d\left( \overrightarrow{\alpha },\overrightarrow{\beta }%
\right) \right\vert \leq 1-\delta \frac{\left\vert \overrightarrow{\alpha }%
\bigtriangleup \overrightarrow{\beta }\right\vert }{\left\vert 
\overrightarrow{\alpha }\right\vert +\left\vert \overrightarrow{\beta }%
\right\vert },  \label{distance-comparison}
\end{equation}%
for any $\overrightarrow{\alpha },\overrightarrow{\beta }\in \mathbb{Z}%
_{+}^{n}$.
\end{proposition}

\begin{proof}
Inequality $\left( \ref{distance-comparison}\right) $ is equivalent to the
following combinatorial inequality%
\begin{equation}
\Pi _{r=1}^{s}\frac{\left( 2\sigma _{r}-1\right) \left( 2\sigma
_{r}-3\right) \cdots \left( 2\sigma _{r}-2l_{r}+1\right) }{\left( 2\sigma
_{r}+1\right) \left( 2\sigma _{r}+3\right) \cdots \left( 2\sigma
_{r}+2l_{r}-1\right) }\leq 1-\delta \frac{\Sigma _{r=1}^{s}l_{r}}{\Sigma
_{r=1}^{s}\sigma _{r}},  \label{distance-compare}
\end{equation}%
where $\sigma _{r}\geq l_{r}\geq 0$ are integers. We first prove $\left( \ref%
{distance-compare}\right) $ when $s=1$. Let $f\left( x\right) =\ln \left( 
\frac{1-x}{1+x}\right) $, then $\acute{f}^{\prime \prime }\left( x\right)
\leq 0$ on $[0,1)$. Therefore%
\begin{eqnarray}
\frac{f\left( \frac{1}{2\sigma _{r}}\right) +\cdots +f\left( \frac{2l_{r}-1}{%
2\sigma _{r}}\right) }{l_{r}} &\leq &f\left( \frac{\frac{1}{2\sigma _{r}}%
+\cdots +\frac{2l_{r}-1}{2\sigma _{r}}}{l_{r}}\right) =f\left( \frac{l_{r}}{%
2\sigma _{r}}\right) ,\text{ i.e.}  \notag \\
\left( \frac{1-\frac{1}{2\sigma _{r}}}{1+\frac{1}{2\sigma _{r}}}\right)
\cdots \left( \frac{1-\frac{2l_{r}-1}{2\sigma _{r}}}{1+\frac{2l_{r}-1}{%
2\sigma _{r}}}\right) &\leq &\left( \frac{1-\frac{l_{r}}{2\sigma _{r}}}{1+%
\frac{l_{r}}{2\sigma _{r}}}\right) ^{l_{r}}\leq \frac{1-\frac{l_{r}}{2\sigma
_{r}}}{1+\frac{l_{r}}{2\sigma _{r}}}\leq 1-\frac{l_{r}}{2\sigma _{r}}.
\label{single-product}
\end{eqnarray}%
For $s>1$, using $\left( \ref{single-product}\right) $ we have%
\begin{eqnarray*}
&&\Pi _{r=1}^{s}\frac{\left( 2\sigma _{r}-1\right) \left( 2\sigma
_{r}-3\right) \cdots \left( 2\sigma _{r}-2l_{r}+1\right) }{\left( 2\sigma
_{r}+1\right) \left( 2\sigma _{r}+3\right) \cdots \left( 2\sigma
_{r}+2l_{r}-1\right) } \\
&\leq &\Pi _{r=1}^{s}\left( 1-\frac{l_{s}}{2\sigma _{s}}\right) \leq \left( 
\frac{\Sigma _{r=1}^{s}\left( 1-\frac{l_{r}}{2\sigma _{r}}\right) }{s}%
\right) ^{s}\text{ (arithmetic-geometric inequality)} \\
&=&\left( 1-\frac{1}{2s}\Sigma _{r=1}^{s}\frac{l_{r}}{\sigma _{r}}\right)
^{s}\leq \left( 1-\frac{1}{2s}\left( \frac{\Sigma _{r=1}^{s}l_{r}}{\Sigma
_{r=1}^{s}\sigma _{r}}\right) \right) ^{s}\text{ (since }\frac{a}{b}+\frac{c%
}{d}\geq \frac{a+c}{b+d}\text{ for }a,b,c,d>0\text{)} \\
&\leq &1-\delta \frac{\Sigma _{r=1}^{s}l_{r}}{\Sigma _{r=1}^{s}\sigma _{r}}%
\text{, }
\end{eqnarray*}%
$\allowbreak $where the last inequality is because there exists a $\delta >0$
such that $\left( 1-\frac{1}{2s}x\right) ^{s}\leq 1-\delta x$ on $\left[ 0,1%
\right] $.
\end{proof}

Inequality $\left( \ref{distance-comparison}\right) $ gives a quantitative
control of the linear independence of the derivative vectors $D^{%
\overrightarrow{\alpha }}\Phi _{t}\left( x\right) $ and $D^{\overrightarrow{%
\beta }}\Phi _{t}\left( x\right) $ as $t\rightarrow 0_{+}$. When $\left\vert 
\overrightarrow{\alpha }\right\vert =\left\vert \overrightarrow{\beta }%
\right\vert =2$, the \emph{uniform} linear independence of $D^{%
\overrightarrow{\alpha }}\Phi _{t}\left( x\right) $ and $D^{\overrightarrow{%
\beta }}\Phi _{t}\left( x\right) $ as $t\rightarrow 0_{+}$ played an
important role in \cite{WZ} to construct isometric embeddings by the
implicit function theorem. It will be interesting to see if there are other
geometric problems where the higher-jet linear independence property can
play role.

\bigskip

{\small Ke Zhu}

{\small Department of Mathematics, Harvard University, Cambridge, MA 02138}

{\small Email: kzhu@math.harvard.edu}

\end{document}